\documentclass{amsart}
\usepackage{amssymb}
\usepackage{graphicx}

\usepackage{longtable}

\newcommand{\cn}{\mathcal{C}}
\newcommand{\bcn}{\overline{\mathcal{C}}}
\newcommand{\hm}{\mathcal{H}}

\newcommand{\x}{\mathbf{x}}
\newcommand{\A}{\mathbf{A}}
\newcommand{\indI}{\mathcal{I}}

\newcommand{\isom}{\rho}

\DeclareMathOperator{\Ker}{Ker}

\DeclareMathOperator{\Img}{Im}

\newtheorem{theorem}{Theorem}
\newtheorem{definition}{Definition}

\begin{document}

\title{A volume form on the Khovanov invariant}

\author{Juan Ortiz-Navarro}

\email{juorna@gmail.com}

\address{Department of Mathematics, Hylan Building, University of Rochester,
Rochester, NY, 14627}

\begin{abstract}
The Reidemeister torsion construction can be applied to the chain
complex used to compute the Khovanov homology of a knot or a link.
This defines a volume form on Khovanov homology. The volume form transforms
correctly under Reidemeister moves to give an invariant volume on
the Khovanov homology. In this paper, its construction and invariance
under these moves is demonstrated. Also, some examples of the invariant
are presented for particular choices for the bases of homology groups
to obtain a numerical invariant of knots and links. In these examples,
the algebraic torsion seen in the Khovanov chain complex when homology
is computed over $\mathbb{Z}$ is recovered.

\end{abstract}
\maketitle

\section{Introduction}

In the 1930s, W. Franz \cite{Franz1935} and K. Reidemeister \cite{Reidemeister1935}
introduced the theory of torsion (also called R-torsion) of a cellular
complex in their study of lens spaces. The lens spaces $L(p,q)$ for
$p$ fixed have the same homology groups but they are not all homeomorphic.
In some cases they are not even homotopy equivalent. The Reidemeister
torsion captures some of the interactions that happen under the radar
and helps distinguish between many of these spaces. Since Khovanov
Homology (cohomology) is defined on a cochain complex, we will talk
about Reidemeister torsion on cochain complexes for the remainder
of this paper.

Reidemeister Torsion is defined for cochain complexes over a field
$\mathbb{F}$ (or, more general, over an asociative ring with multiplicative
identity). Over a field, cochain groups become vector spaces. Given
an $n$ dimensional vector space $V$, a multilinear function $T:V^{k}\rightarrow\mathcal{\mathbb{F}}$
is called a $k$-$\mathit{tensor}$. The set of all alternating $k$-tensors,
denoted $\Lambda^{k}(V)$, is a vector space over $\mathcal{\mathbb{F}}$
of dimension $\binom{n}{k}$. A nonzero element of $\Lambda^{n}(V)$
is called a $\mathit{volume}$ $\mathit{form}$ for the space $V$.

Reidemeister torsion for a cochain complex $C$ defines a volume form
on the space:\begin{equation}
(C^{0})^{*}\oplus C^{1}\oplus(C^{2})^{*}\oplus\dotsc C^{m}(\text{or }(C^{m})^{*})\end{equation}
 where the last component of this sum (whether we add the vector or
its dual) depends upon the parity of $m$. 

Each basis for a vector space yields a volume form in such a way that
if we change basis, the volume form is transformed by the determinant
of the change of basis matrix between the two bases. Having a volume
form on the space described above and denoting cohomology groups by
$H^{r}$, one can define a volume form on the space:\begin{equation}
(H^{0})^{*}\oplus H^{1}\oplus(H^{2})^{*}\oplus\dotsc H^{m}(\text{or }(H^{m})^{*})\end{equation}
In this sense, for acyclic cochain complexes, Reidemeister torsion
is a volume form on the $0$ vector space. In this case $\Lambda^{0}(0)=\mathcal{\mathbb{F}}$,
i.e., we obtain an element of $\mathcal{\mathbb{F}}$ that depends
only on the bases specified for the chain groups. For non-acyclic
cochain complexes, the torsion depends on the bases for the cochain
groups as well as the bases specified for the homology groups.

In 2001, Mikhail Khovanov \cite{Khovanov1999} presented a new theory,
which assigns homology groups to a knot diagram. His theory sparked
new questions in knot theory, as well as, proposed new ways of attacking
old problems in topology. Khovanov theory constructs a complex from
a diagram of a link. The cohomology groups obtained from this complex
are invariant under Reidemeister moves in the diagram, which makes
them a topological invariant of the link or knot under study. Cochain
groups in this complex are made of tensor products of a graded algebra
over a ring and coboundary operators arise from operations in the
algebra. The complex itself decomposes as subcomplexes that preserve
the grading in the algebra. From this point of view, Khovanov homology
recovers other invariants of links, such as the well-known Jones polynomial
introduced by Jones in \cite{Jones1985}.

In this paper, we construct a topological invariant for knots and
links using Reidemeister torsion, which gives a well defined volume
form for the Khovanov homology. We showed this by demonstrating the
volume form is preserved under Reidemeister moves by looking at the
cochain maps inducing isomorphisms in homology for each of this moves,
as presented in \cite{Khovanov1999}. We also demonstrate that, for
acylic subcomplexes, the Reidemeister torsion gives an invariant number
for knots and links.

In Section \ref{sec:R-torsion}, we present a brief introduction to
Reidemeister torsion. Section \ref{sec:Khovanov-Homology} serves
as an introduction to Khovanov homology. In Section \ref{sec:main},
we demonstrate the volume form on Khovanov homology. In Section \ref{sec:Examples},
we calculate the form for a simple example and show tables for some
knots and links for a very special choice of bases for the cohomology
groups.

\section{R-torsion}

\label{sec:R-torsion}

As defined in \cite{Turaev2001}.

\subsection{Definition for acyclic complexes}

Let $\mathbb{F}$ be a field and $D$ be a finite-dimensional vector
space over $\mathbb{F}$. Suppose that $\dim D=k$ and pick two (ordered)
bases $b=(b_{1},\dotsc,b_{k})$ and $c=(c_{1},\dotsc,c_{k})$ of $D$.
Let \[
b_{j}=\sum_{i=1}^{k}a_{ij}c_{i},\qquad j=1,\dotsc,k,\]
 The matrix $(a_{ij})_{i,j=1,\dotsc,k}$ is called the \textit{transition
matrix} between the basis $b$ and $c$ and it is a nondegenerate
$(k\times k)$ - matrix over $\mathbb{F}$. We write \[
[b/c]=\det(a_{ij})\thickspace\in\mathbb{F}^{*}=\mathbb{F}-0\]
 One can show that the relation, $b\sim c$ if and only if $[b/c]=1$,
is an equivalence relation.

Let $C$, $D$ and $E$ be vector spaces over $\mathbb{F}$. Let\[
0\rightarrow C\xrightarrow{i}D\xrightarrow{\pi}E\rightarrow0\]
be a \textit{}short exact sequence of vector spaces. Then $\dim D=\dim C+\dim E$.
Let $c=(c_{1},\dotsc,c_{k})$ be a basis for $C$, $d=(d_{1},\dotsc,d_{k})$
be a basis for $D$ and $e=(e_{1},\dotsc,e_{l})$ be a basis for $E$.
Since $\beta$ is surjective, we may lift each $e_{i}$ to some $\tilde{e}_{i}\in D$,
such that $\beta(\tilde{e}_{i})=e_{i}$. We will call $\tilde{e}=(\tilde{e}_{1},\dotsc,\tilde{e}_{l})$
a \textit{pullback} for $e$. We set \[
c\medspace\tilde{e}=(c_{1},\dotsc,c_{k},\tilde{e}_{1},\dotsc,\tilde{e}_{l}).\]
 Then $c\medspace\tilde{e}$ is a basis for $D$.

Let \[
C=(0\rightarrow C^{0}\xrightarrow{\partial^{0}}C^{1}\rightarrow\dotsb\xrightarrow{\partial^{m-2}}C^{m-1}\xrightarrow{\partial^{m-1}}C^{m}\rightarrow0)\]
 be an acylcic based cochain complex over $\mathbb{F}$. Set $B^{r-1}=\Img(\partial^{r-1}:C^{r-1}\rightarrow C^{r})\subset C^{r}$.
Since $C$ is acyclic, \[
C^{r}/B^{r-1}=C^{r}/\ker(\partial^{r}:C^{r}\rightarrow C^{r+1})\cong im\medspace\partial^{r}=B^{r}.\]
 In other words, the sequence \[
0\rightarrow B^{r-1}\hookrightarrow C^{r}\xrightarrow{\partial^{r}}B^{r}\rightarrow0\]
 is exact. Choose a basis $b^{r}$ of $B^{r}$ for $r=-1,\dotsc,m$.
By the above construction, $b^{r-1}\medspace\tilde{b}^{r}$ is a basis
of $C^{r}$, which can be compared with the basis $c^{r}$ of $C^{r}$.

\begin{definition} The Reidemeister torsion of $C$ is \[
\tau(C)=\biggl\lvert\prod_{r=0}^{m}[b^{r-1}\medspace\tilde{b}^{r}/c^{r}]^{(-1)^{r+1}}\biggr\rvert\in\mathbb{F}^{*}.\]
 \end{definition}

\textbf{Remarks:} (see \cite{Turaev2001}) 

\begin{itemize}
\item $\tau(C)$ does NOT depend on the choice for $b^{r}$ and its pullback
$\tilde{b}^{r}$. 
\item $\tau(C)$ DOES depend on $c^{r}$, which is called the \textit{distinguished
basis} for the cochain group $C^{r}$. 
\item If another basis for the cochain groups is equivalent to the distinguished
basis then $\tau(C)$ is the same for both bases. Indeed, if $C'$
is the same acyclic cochain complex $C$ based $c'=(c'^{0},\dotsc,c'^{m})$,
then\begin{equation*}
\tau(C')=\tau(C)\prod^m_{i=0} [c^i / c'^i]^{(-1)^{i+1}}.
\end{equation*}
\end{itemize}

\subsection{R-torsion as a volume form}

R-torsion can be understood as a volume form on the space:\begin{equation}
(C^0)^* \oplus (C^1) \oplus (C^2)^* \oplus \dotsb \oplus (C^m) \text{ (or } (C^m)^* \text{)}
\end{equation} where ${(C}^{r})^{*}=Hom(C^{r},\mathbb{F})$ is the dual of $C^{r}$,
and the last component $(C^{m})$ (or its dual) depends on the parity
of $m$.

Consider the cochain complex $(D,\Delta)$ (not necessarily acyclic),
then by the $1$st isomorphism theorem: \[
0\rightarrow\ker\Delta^{i}\hookrightarrow D^{i}\xrightarrow{\Delta^{i}}\Img\Delta^{i}\rightarrow0\]
 is exact. Moreover, \[
0\rightarrow B^{i-1}\hookrightarrow\ker\Delta^{i}\xrightarrow{\pi}H^{i}\rightarrow0\]
 is also exact. Thus, having basis $b^{i}$ for $B^{i}=\Img\Delta^{i}$
(as before) and $[h^{i}]$ for $H^{i}$ for every $i$, we obtain
the basis $[b^{i-1}\medspace h^{i}]$ for $\ker\Delta^{i}$, where
$\pi(h^{i})=[h^{i}]$. Furthermore, $[b^{i-1}\medspace h^{i}\medspace\tilde{b}^{i}]$
is a basis for $D^{i}$.

\begin{definition} The Reidemeister torsion of $D$ is \[
\tau(D)=\biggl\lvert\prod_{r=0}^{m}[b^{r-1}\medspace h^{r}\medspace\tilde{b}^{r}/c^{r}]^{(-1)^{r+1}}\biggr\rvert\in\mathbb{F}^{*}.\]
 \end{definition}

\textbf{Remarks:} (see \cite{Turaev2001}) 

\begin{itemize}
\item $\tau(C)$ does NOT depend on the choice for $b^{r}$ and its pullback
$\tilde{b}^{r}$. 
\item Using the quotient map $\ker\Delta^{i}\xrightarrow{\pi}{H^{i}}$,
and a pullback for a cohomology basis, one can think of cohomology
as lying in the cochain group and the restriction of the volume form
on cohomology makes R-torsion a volume form for the vector space $(H^{0})^{*}\oplus H^{1}\oplus(H^{2})^{*}\oplus\dotsc H^{m}(\text{or }(H^{m})^{*})$,
(depending on the parity of $m$). 
\end{itemize}

\subsection{Mapping Cone}

Suppose there is a cochain map $\varphi:(C,\partial_{C})\rightarrow(D,\partial_{D})$
between two cochain complexes over a field. \begin{definition} The
mapping cone $m(\varphi)$ is a cochain complex $(E,\delta)$ with
cochain groups: \[
E^{r}=C^{r}\oplus D^{r-1}\]
 and coboundary operator: \[
\delta^{r}:E^{r}\rightarrow E^{r+1}\]
 so that for $c\in C^{r}$ and $d\in D^{r-1}$ \[
\delta^{r}(c,d)=\big{(}\partial_{C}^{r}(c),\partial_{D}^{r-1}(d)+(-1)^{r}\varphi_{r}(c)\big{)}\]
 or \[
\delta^{r}=\left(\begin{matrix}\partial_{C}^{r} & 0\\
(-1)^{r}\varphi^{r} & \partial_{D}^{r-1}\end{matrix}\right)\]
where this matrix acts on a vector whose first block represents an
element in $C^{r}$ and the second block represents an element in
$D^{r-1}$. (Similar notation should be understood similary through
the rest of this paper.) \end{definition} 

Note that the distinguished bases for the cochain groups of $C$ and
$D$ combine to give distinguished bases for the cochain groups of
$m(\varphi)$.

\subsection{Quasi-isomorphisms}

A cochain map $\varphi:(C,\partial_{C})\rightarrow(D,\partial_{D})$
between two cochain complexes over a field induces a well defined
map in cohomology, $\varphi_{*}^{r}:H^{r}(C)\rightarrow H^{r}(D)$,
\[
\varphi_{*}^{r}([h])=[\varphi^{r}(h)]\qquad\text{for}\quad[h]\in H^{r}(C)\]
 If this map is an isomorphism for every $r$, then the cochain map
is called a \textit{quasi-isomorphism}. Note that the mapping cone
of a quasi-isomorphism is acyclic. The short exact sequence: \[
0\rightarrow D^{r-1}\hookrightarrow D^{r-1}\oplus C^{r}\xrightarrow{\pi}C^{r}\rightarrow0\]
 induces a long exact sequence in cohomology: \[
\dotsb\rightarrow H^{r-1}(D)\rightarrow H^{r}(E)\rightarrow H^{r}(C)\xrightarrow{\varphi_{*}^{r}}H^{r}(D)\rightarrow H^{r+1}(E)\rightarrow\dotsb\]
 Since $\varphi_{*}^{r}:H^{r}(C)\cong H^{r}(D)$, one can decompose
the sequence to obtain: \[
0\rightarrow H^{r}(E)\rightarrow0\]
 to get $H^{r}(E)=0$ for every $r$.

In \cite{ChungLing2006}, Chung and Lin presented a theory for the
torsion of quasi-isomorphisms.

\begin{definition} The torsion $\tau(\varphi)$ of the quasi-ismorphism
$\varphi:(C,\partial_{C})\rightarrow(D,\partial_{D})$ is \begin{equation}
\tau(\varphi)=\biggl\lvert\prod_{r=0}^{m}\frac{[b^{r-1}\medspace h^{r}\medspace\tilde{b}^{r}/c^{r}]}{[b'^{r-1}\medspace\varphi_{*}^{r}(h^{r})\medspace\tilde{b'}^{r}/c^{r}]}^{(-1)^{r+1}}\biggr\rvert\label{torsion-quasi}\end{equation}
 where $b^{i}$ and $b'^{i}$ are bases for $\Img\partial_{C}$ and
$\Img\partial_{D}$ respectively. \end{definition}

\textbf{Remarks:} (see \cite{ChungLing2006}) 

\begin{itemize}
\item $\tau(\varphi)$ does NOT depend on the choice of basis for the cohomological
groups. 
\item Their definition coincides with the torsion of the acyclic mapping
cone of the quasi-isomorphism and is independent of the choice of
bases for the cohomology groups. 
\end{itemize}

\subsection{A special quasi-isomorphism}

Consider a quasi-isomorphism $\varphi:(C,\partial_{C})\rightarrow(D,\partial_{D})$
such that its matrix representation according to the distinguished
bases for all cochain groups for $C$ and $D$ is given by identity
matrices (not necessarily the identity map, since D and C may not
have the same cochain groups). Then, the coboundary operator $\delta^{r}$
of its mapping cone looks like: \begin{equation}
\delta^{r}=\left(\begin{matrix}\partial_{C}^{r} & 0\\
(-1)^{r}Id & \partial_{D}^{r-1}\end{matrix}\right)\end{equation}

Now: \[
\Img\delta^{r}=\Img\partial_{C}^{r}\bigoplus D^{r}\]
 Then $\mathcal{B}^{r}=(b_{C}^{r},0)\cup(0,d^{r})$ is a basis for
$\Img\delta^{r}$, where $b_{C}^{r}$ is a basis for $\Img\partial_{C}^{r}$
and $d^{r}$ is the distinguished basis for $D^{r}$. Then a pullback
is given by $\tilde{\mathcal{B}^{r}}=c^{r}$, where $c^{r}$ is the
distinguished basis for $C^{r}$. Therefore: \begin{equation}
[\mathcal{B}^{r-1}\medspace\tilde{\mathcal{B}^{r}}/c^{r}\oplus d^{r-1}]=\left\lvert\begin{matrix}* & Id\\
Id & 0\end{matrix}\right\rvert=\ 1\label{torsion-id-quasi}\end{equation}
 Thus the Reidemeister torsion of the quasi-isomorphism is $1$. Note
that in this case, using the isomorphism $\varphi_{*}$ and its dual
$\varphi_{*}^{*}$: \begin{equation}
(H_{C}^{0})^{*}\oplus H_{C}^{1}\oplus(H_{C}^{2})^{*}\oplus\dotsc H_{C}^{m}(\text{or }(H_{C}^{m})^{*})\cong(H_{D}^{0})^{*}\oplus H_{D}^{1}\oplus(H_{D}^{2})^{*}\oplus\dotsc H_{D}^{m}(\text{or }(H_{D}^{m})^{*})\end{equation}
 and by \eqref{torsion-quasi}, the volume form is preserved on this
space generated from the cohomology groups.

\section{The Khovanov Chain Complex}

\label{sec:Khovanov-Homology}

As defined in \cite{Bar-Natan2002}.

All links are oriented in an oriented Euclidean space. We will present
them using projections to the plane. Let $D$ be a diagram of a link
$L$, $\mathcal{X}$ be the set of crossings of $D$, $n=|\mathcal{X}|$.
Let us number the elements of $\mathcal{X}$ from $1$ to $n$ and
write $n=n_{+}+n_{-}$ where $n_{+}$ ($n_{-}$) is the number of
right-handed (left-handed) crossings in $\mathcal{X}$.

\subsection{Spaces}

\begin{definition} Let $W=\bigoplus_{m}W_{m}$ be a graded vector
space with homogeneous components $\{ W_{m}\}$. The graded dimension
of $W$ is the power series $qdim\medspace W:=\sum_{m}q^{m}\dim W_{m}$.
\end{definition}

\begin{definition} If $W=\bigoplus_{m}W_{m}$ is a graded vector
space, we set $W\{ l\}_{m}:=W_{m-l}$, so that $qdim\medspace W\{ l\}=q^{l}\medspace qdim\medspace W$.
\end{definition}

\begin{definition} If $\bar{C}$ is a cochain complex $\ldots\to\bar{C}^{r}\overset{d^{r}}{\to}\bar{C}^{r+1}\ldots$
of vector spaces, and if $C=\bar{C}[s]$, then $C^{r}=\bar{C}^{r-s}$
(differentials also shifted). \end{definition}

Let $V$ be the graded vector space with two basis elements $x$ and
$1$ whose degrees are $-1$ and $1$ respectively, so that $qdimV=q+q^{-1}$.
For every vertex $\alpha\in\{0,1\}^{\mathcal{X}}$ of the cube $\{0,1\}^{\mathcal{X}}$,
we associate the graded vector space $V_{\alpha}(D):=V^{\otimes k}\{ r\}$,
where $k$ is the number of cycles in the smoothing of $D$ corresponding
to $\alpha$ and $r$ is the height $|\alpha|=\sum_{i}\alpha_{i}$
of $\alpha$. We set the $r$th cochain group $\lVert D\rVert^{r}$
(for $0\leq r\leq n$) to be the direct sum of all the vector spaces
at height $r$: $\lVert L\rVert^{r}:=\bigoplus_{\alpha:r=|\alpha|}V_{\alpha}(D)$.
Finally, we set $C(D):=\lVert D\rVert[-n_{-}]\{ n_{+}-2n_{-}\}$.

\begin{figure}[h]

\begin{centering}\includegraphics[scale=0.6]{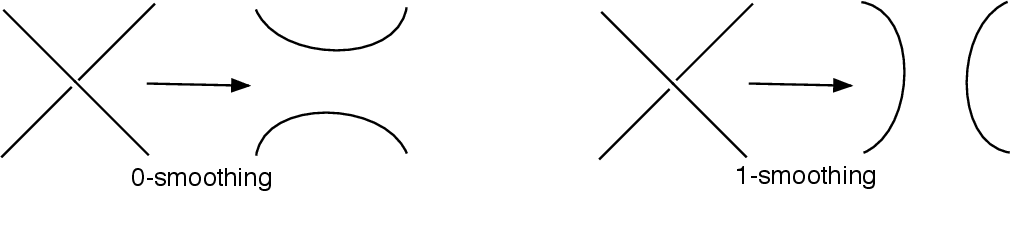} \par\end{centering}

\caption{Smoothing a crossing}
\end{figure}

\subsection{Distinguished Basis}

Each cochain group in the Khovanov complex has an ordered basis in
the following way. The vector space $V$ has $\{1,x\}$ as a basis.
Now, consider the tensor product of $r$ copies of $V$. Using reverse
lexicographic order, we get the ordered basis:

\begin{equation}
\begin{split}1\otimes1\otimes1 & \otimes\dotsc\otimes1\\
x\otimes1\otimes1 & \otimes\dotsc\otimes1\\
1\otimes x\otimes1 & \otimes\dotsc\otimes1\\
x\otimes x\otimes1 & \otimes\dotsc\otimes1\\
 & \vdots\\
1\otimes x\otimes x & \otimes\dotsc\otimes x\\
x\otimes x\otimes x & \otimes\dotsc\otimes x\end{split}
\end{equation}

For every vertex $\alpha$, $V_{\alpha}(D)$ is assigned such a basis,
so that the cochain group $\lVert D\rVert^{r}$ has basis consisting
of the injection of each of these bases. This basis will be the distinguished
basis for the $r$th cochain group.

\subsection{Maps}

The space $V_{\alpha}(D)$ on each vertex $\alpha$ has as many tensor
factors as there are components in the smoothing $S_{\alpha}$. Thus,
we put these tensor factors in $V_{\alpha}$ and cycles in $S_{\alpha}$
in bijective correspondence. Each edge $\xi$ of the cube maps the
vector spaces at its ends. For any edge $\xi$, the smoothing at the
tail of $\xi$ differs from the smoothing at the head of $\xi$ slightly:
either two of the components merge into one or one of the components
splits in two. So for any $\xi$, we set $d_{\xi}$ to be the identity
on the tensor factors corresponding to the components that don't participate,
and then we complete the definition of $d_{\xi}$ using two linear
maps: $m:V\otimes V\to V$ and $\Delta:V\to V\otimes V$ defined as
follows: \begin{align}
\big(V\otimes V\overset{m}{\rightarrow}V\big) & \quad & m: & \begin{cases}
1\otimes x\mapsto x & 1\otimes1\mapsto1\\
x\otimes1\mapsto x & x\otimes x\mapsto0\end{cases}\\
\big(V\overset{\Delta}{\rightarrow}V\otimes V\big) & \quad & \Delta: & \begin{cases}
1\mapsto1\otimes x+x\otimes1\\
x\mapsto x\otimes x\end{cases}\end{align}

The height $|\xi|$ of an edge $\xi$ is defined to be the height
of its tail. Hence, if the maps on the edges are called $\partial_{\xi}$,
then the vertical collapse of the cube to a cochain complex becomes
$\partial^{r}:=\sum_{|\xi|=r}(-1)^{\xi}\partial_{\xi}$, where $(-1)^{\xi}:=(-1)^{\sum_{i<j}\xi_{i}}$,
and $j$ is the location where the smoothings at the head and tail
of the edge differ.

\subsection{Khovanov Invariant}

Let $\hm^{r}(D)$ denote the $r$th cohomology group of $C(D)$. It
is a graded vector space. Let\begin{equation*}
Kh(L):=\sum_r t^r \text{qdim} \hm ^r(D)
\end{equation*}

\begin{theorem}
$\hm ^r(L)$ is an invariant for the link $L$ and $Kh(L)$ is a link invariant that specializes to the Jones polynomial at $t=-1$.
\end{theorem}

see \cite{Khovanov1999}.

\section{The volume form}

\label{sec:main}

\begin{theorem} The Khovanov homology has a volume form which is
invariant for knots and links. \end{theorem}

In this (long) section, invariance of the volume form is demonstrated
under each of the Reidemeister moves. To obtain this, it will be shown
that for each of the moves the mapping cone of the quasi-isomorphisms
relating each side of the move has Reidemeister Torsion $1$. 

As in \cite{Lee2005}. $\bcn(D)$ will denote the complex assign to
a diagram of a link prior to any shifts, i.e., \[
\cn(D)=\bcn(D)[-n_{-}]\{ n_{+}-2n_{-}\}\]

\subsection{Reidemeister I}

The Reidemeister I move has two variants.

\begin{figure}[h]
\begin{centering}\includegraphics[scale=0.6]{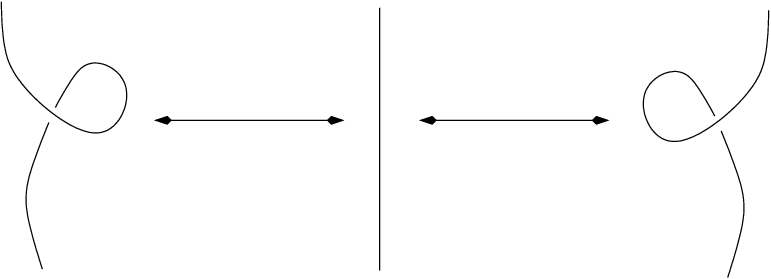} \par\end{centering}

\caption{Reidemeister I move}
\end{figure}

The left side (the negative kink) can be obtained by a series of Reidemeister
I (positive kink), Reidemeister II and Reidemeister III moves by the
Whitney trick. %
\begin{figure}[h]

\begin{centering}\includegraphics[scale=0.6]{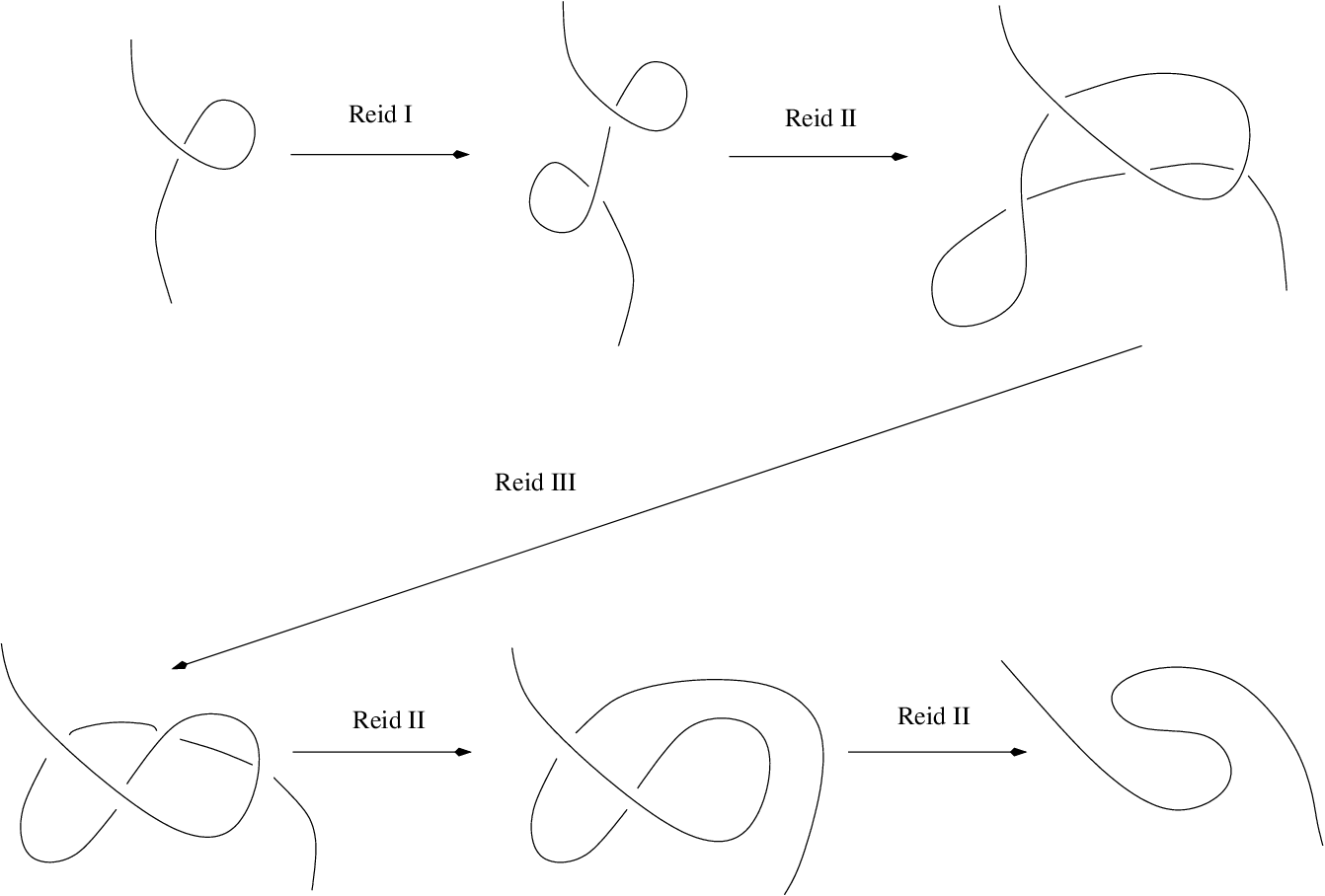} \par\end{centering}

\caption{Whitney trick}
\end{figure}
Hence, invariance for negative kinks follows from the Reidemeister
I (positive), Reidemeister II, and Reidemeister III moves.

\begin{figure}[h]

\begin{centering}\includegraphics[scale=0.6]{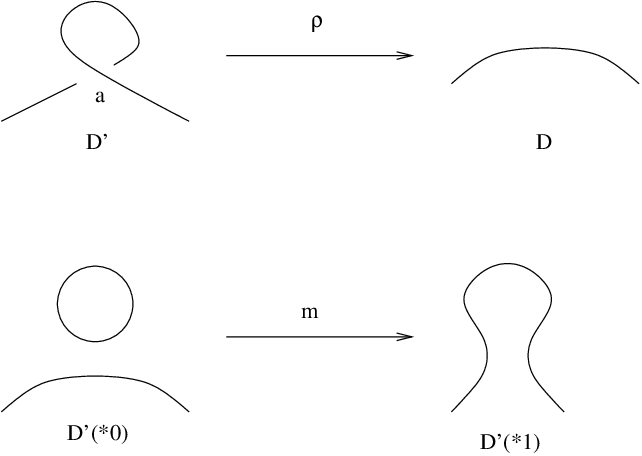} \par\end{centering}

\caption{Reidemeister I Invariance}
\end{figure}

Let $a$ be the crossing which appears only in $D'$. The set $\indI'$
of crossings of $D'$ is $\indI$, the set of crossings of $D$, followed
by $a$ as an ordered $|\indI'|$-tuple. Let $D'(*0)$ and $D'(*1)$
denote $D'$ with only its last crossing (that is $a$) resolved to
its 0- and 1-resolutions, respectively.

Consider $\partial'_{0\rightarrow1}:\bcn(D'(*0))\rightarrow\bcn(D'(*1))[-1]\{-1\}$,
which is always multiplication and define \[
X_{1}=\Ker\partial'_{0\rightarrow1}\]
 and \[
X_{2}=\{ y\otimes  1 + z|y\in\bcn(D),z\in\bcn(D'(*1))[-1]\{-1\}\}\textrm{ .}\]
$X_{1}$ and $X_{2}$ are subcomplexes of $\bcn(D')$. $\bcn(D')$
decomposes as $X_{1}\oplus X_{2}$ as a chain complex. (see \cite{Khovanov1999})

Now, consider the $r$-th cochain group of the complex $\bcn(D')^{r}$.
This space has a distinguished basis coming from the construction
of the complex as explained in Section (\ref{sec:Khovanov-Homology}).

As a basis for $X_{1}^{r}=\Ker\partial'{}_{0\rightarrow1}^{r}$, we
use $\{\dotsc\otimes1\otimes x-\dots\otimes x\otimes1,\dotsc\otimes x\otimes x\}$
and, for $X_{2}^{r}$, we use the subbasis of the distinguished basis
consisting of all basis elements with a $1$ in the last tensor for
components corresponding to states with a positive smoothing at the
crossing $a$ and all the basis elements for components corresponding
to states with a negative smoothing at the crossing $a$.

Using this ordered basis, the change of basis matrix between the basis
coming from $X_{1}^{r}$ and $X_{2}^{r}$ and the distinguished basis
for $\bcn(D')^{r}$ consists of diagonal blocks of the form: \begin{equation}
\left(\begin{tabular}{cccc}
 0  &  0  &  1  &  0 \\
-1  &  0  &  0  &  1 \\
1  &  0  &  0  &  0 \\
0  &  1  &  0  &  0 \end{tabular}\right)\end{equation}
for states with a positive marker at the crossing $a$. The states
with a negative marker at this crossing belong to the subcomplex $X_{2}$
and the basis is unchanged. In any case, the bases for the cochain
groups coming from the decomposition of the complex as subcomplexes
$X_{1}$ and $X_{2}$ are equivalent to the distinguished bases.

The subcomplex $X_{2}=\{ y\otimes 1+z|y\in\bcn(D),z\in\bcn(D'(*1))[-1]\{-1\}\}$
is acyclic. (see \cite{Khovanov1999}) Let's find its Reidemeister
torsion. We only need to find basis for the image of the boundary
operator and a pullback for this basis.

\begin{figure}[h]

\begin{centering}\includegraphics[scale=0.6]{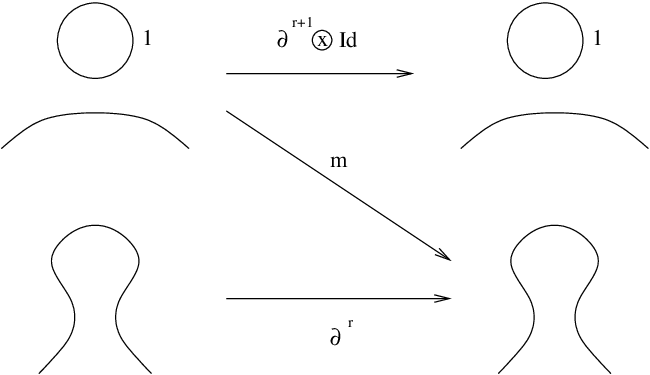} \par\end{centering}

\caption{$\partial'^{r}|_{X_{2}^{r}}:X_{2}^{r}\rightarrow X_{2}^{r+1}$}
\end{figure}

The image of $\partial'^{r}|_{X_{2}^{r}}$ consists of: \[
{\{(\partial^{r+1}(y)\otimes1,y),\medspace y\in\bcn(D)^{r+1}\}}\]
Thus, a basis $b'^{r}$ for $\Img\partial'^{r}$ is given by: \[
(b^{r+1}\otimes1,c^{r+1})\]
 where $b^{r+1}$ is a basis for $\partial^{r+1}$ on $(\bcn(D),\partial)$
and $c^{r+1}$ is the distinguished basis for $\bcn(D(*1))^{r+1}$.

A pullback $\tilde{b}'^{r}$ for this basis is $(c^{r+1}\otimes1,0)$.
Hence $[b'^{r-1}\medspace\tilde{b'^{r}}/\mathcal{X}_{2}^{r}]$ is
the determinant of the change of basis matrix between the ordered
basis for $b'^{r-1}\medspace\tilde{b'^{r}}$ as described, and $\mathcal{X}_{2}^{r}$,
the previously described basis for $X_{2}^{r}$: \begin{equation}
\left(\begin{tabular}{cc}
 $b^{r}\otimes1$  &  Id \\
Id  &  0 \end{tabular}\right)\end{equation}
 which has determinant $\pm1$.

The map \begin{eqnarray*}
\isom: & X_{1} & \longrightarrow\;\bcn(D)\{1\}\\
 & y\otimes 1+z\otimes\x & \longmapsto\; z\end{eqnarray*}
 induces an isomorphism between $\hm(D')$ and $\hm(D)$. (see \cite{Khovanov1999})
With the described bases, the matrix representation of the cochain
map is the identity matrix for all cochain groups. Hence by \eqref{torsion-id-quasi}
the Reidemeister torsion of the quasi-isomorphism is $1$.

\subsection{Reidemeister II Move}

\begin{figure}[h]

\begin{centering}\includegraphics[scale=0.6]{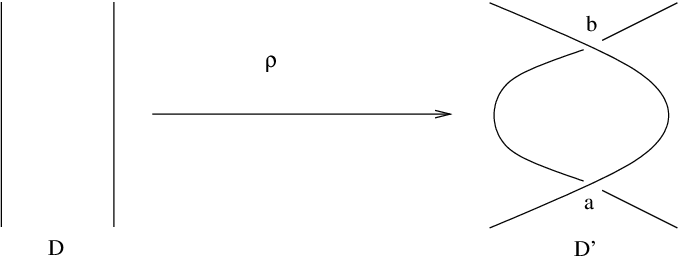} \par\end{centering}

\caption{Reidemeister II invariance}
\end{figure}

As before, the set $\indI'$ of crossings of $D'$ is $\indI$, the
set of crossings of $D$, followed by $a$, then $b$ as an ordered
$|\indI'|$-tuple.

\begin{figure}[h]

\begin{centering}\includegraphics[scale=0.6]{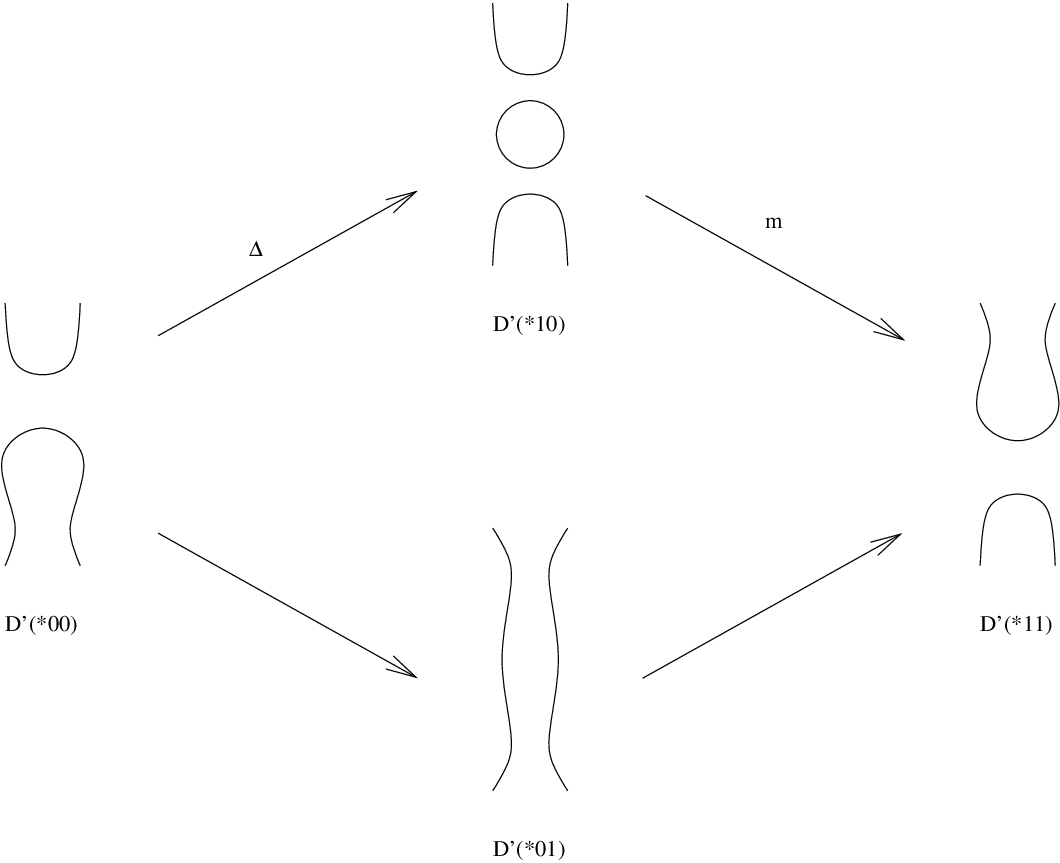} \par\end{centering}

\caption{Complex for the right side $\bcn(D')$}
\end{figure}

Let: \begin{eqnarray*}
X_{1} & = & \{ z+\alpha(z)|z\in\bcn(D'(*01))[-1]\{-1\}\}\\
X_{2} & = & \{ z+\partial'y|z,y\in\bcn(D'(*00))\}\\
X_{3} & = & \{ z+y\otimes 1|z,y\in\bcn(D'(*11))[-2]\{-2\}\}\end{eqnarray*}
 where $\alpha(z)=-\partial'_{01\rightarrow11}(z)\otimes 1\in\bcn(D'(*10))[-1]\{-1\}$.

Then, $\bcn(D')$ is the direct sum of its subcomplexes $X_{1}$,
$X_{2}$, and $X_{3}$. (see \cite{Khovanov1999})

Now consider the $r$-th cochain group of the complex for the left
side , $\bcn(D')^{r}$. This space has a distinguished basis coming
from the construction of the complex. Let's see how the decomposition
of the complex into these subcomplexes affects the torsion. Let's
consider first the case when $\partial'_{01\rightarrow11}=m$. We
can handle this locally since everywhere else the complex is unchanged.

\begin{figure}[h]

\begin{centering}\includegraphics[scale=0.6]{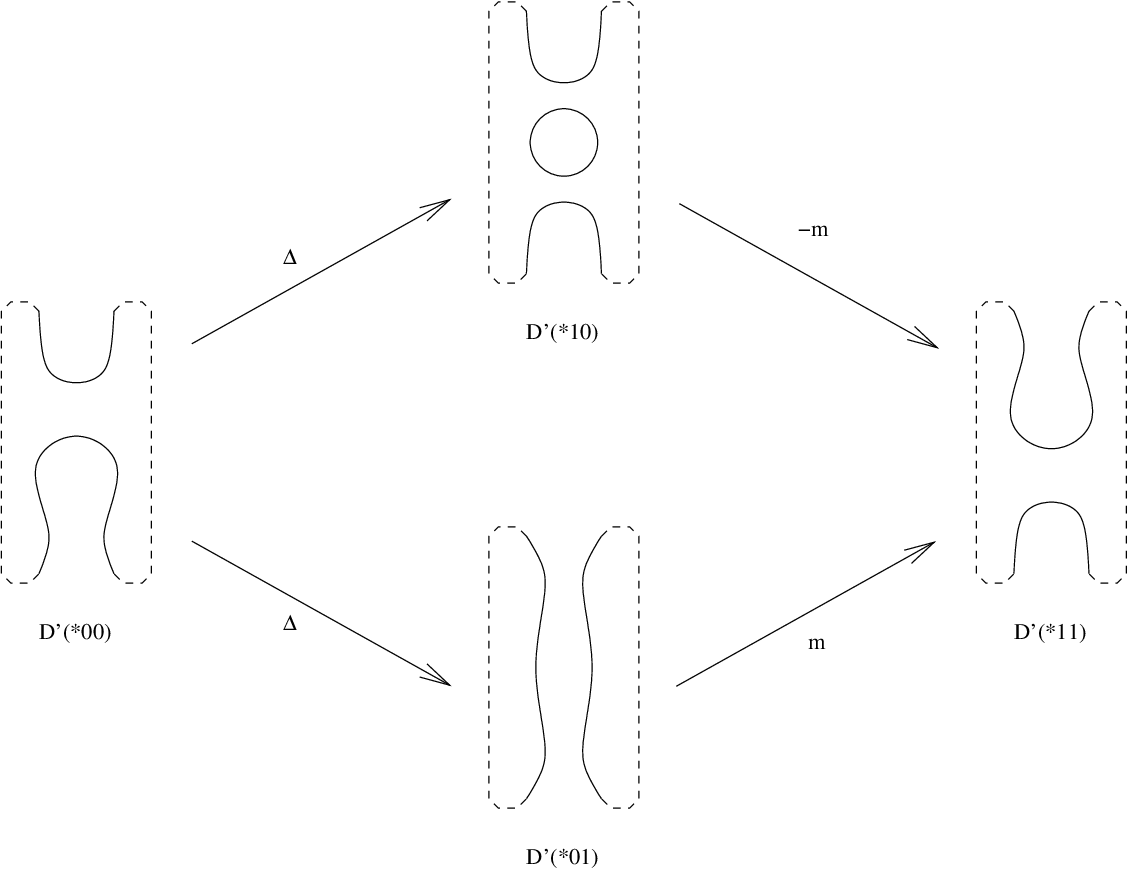} \par\end{centering}

\caption{$\bcn(D')$ for the case when $\partial'_{01\rightarrow11}=m$}
\end{figure}

Since $X_{1}^{r}=\{ z+\alpha(z)|z\in\bcn(D'(*01))[-1]\{-1\}^{r}\}$,
we can choose the basis: \begin{eqnarray*}
(1\otimes1 & , & 1\otimes1)\\
(x\otimes1 & , & x\otimes1)\\
(x\otimes1 & , & 1\otimes x)\\
(0 & , & x\otimes x)\end{eqnarray*}
where the first component is an element of $\bcn(D'(*10))^{r}$ and
the second component lies on $\bcn(D'(*01))^{r}$.

Since $X_{2}^{r}=\{ z+\partial'y|z\in\bcn(D'(*00))^{r},y\in\bcn(D'(*00))^{r-1}\}$,
a basis for the $z$ part (which lies in $\bcn(D'(*00))^{r}$), is
just $\{1,x\}$, hence, unchanged, and for the $\partial'(y)$ part
the basis is \[
\{(1\otimes x+x\otimes1,1\otimes x+x\otimes1),(x\otimes x,x\otimes x)\}\]
 which lies on $\bcn(D'(*10))^{r}\oplus\bcn(D'(*01))^{r}$.

Now, $X_{3}^{r}=\{ z+y\otimes 1|z\in\bcn(D'(*11))[-2]\{-2\}^{r},y\in\bcn(D'(*11))[-2]\{-2\}^{r+1}\}$
has as basis $\{1,x\}$ on $\bcn(D'(*11))^{r}$ and \[
\{(1\otimes1,0),(x\otimes1,0)\}\]
 on $\bcn(D'(*10))^{r}\oplus\bcn(D'(*01))^{r}$.

Thus, the bases for $\bcn(D'(*00))^{r}$ and $\bcn(D'(*11))^{r}$
are the distinguished bases. And for $\bcn(D'(*10))^{r}\oplus\bcn(D'(*01))^{r}$
the change of basis matrix between the the ordered basis coming from
$X_{1}^{r}$, $X_{2}^{r}$, and $X_{3}^{r}$ (in the order presented)
and the distinguished basis for $\bcn(D')^{r}$ is:

\begin{equation}
\left(\begin{tabular}{cccccccc}
 1  &  0  &  0  &  0  &  0  &  0  &  1  &  0 \\
0  &  1  &  1  &  0  &  1  &  0  &  0  &  1 \\
0  &  0  &  0  &  0  &  1  &  0  &  0  &  0 \\
0  &  0  &  0  &  0  &  0  &  1  &  0  &  0 \\
1  &  0  &  0  &  0  &  0  &  0  &  0  &  0 \\
0  &  1  &  0  &  0  &  1  &  0  &  0  &  0 \\
0  &  0  &  1  &  0  &  1  &  0  &  0  &  0 \\
0  &  0  &  0  &  1  &  0  &  1  &  0  &  0 \end{tabular}\right)\end{equation}
which has determinant $\pm1$.

Let's consider now the case when $\partial'_{01\rightarrow11}=\Delta$.
We can handle this locally since everywhere else the complex is unchanged.

\begin{figure}[h]

\begin{centering}\includegraphics[scale=0.6]{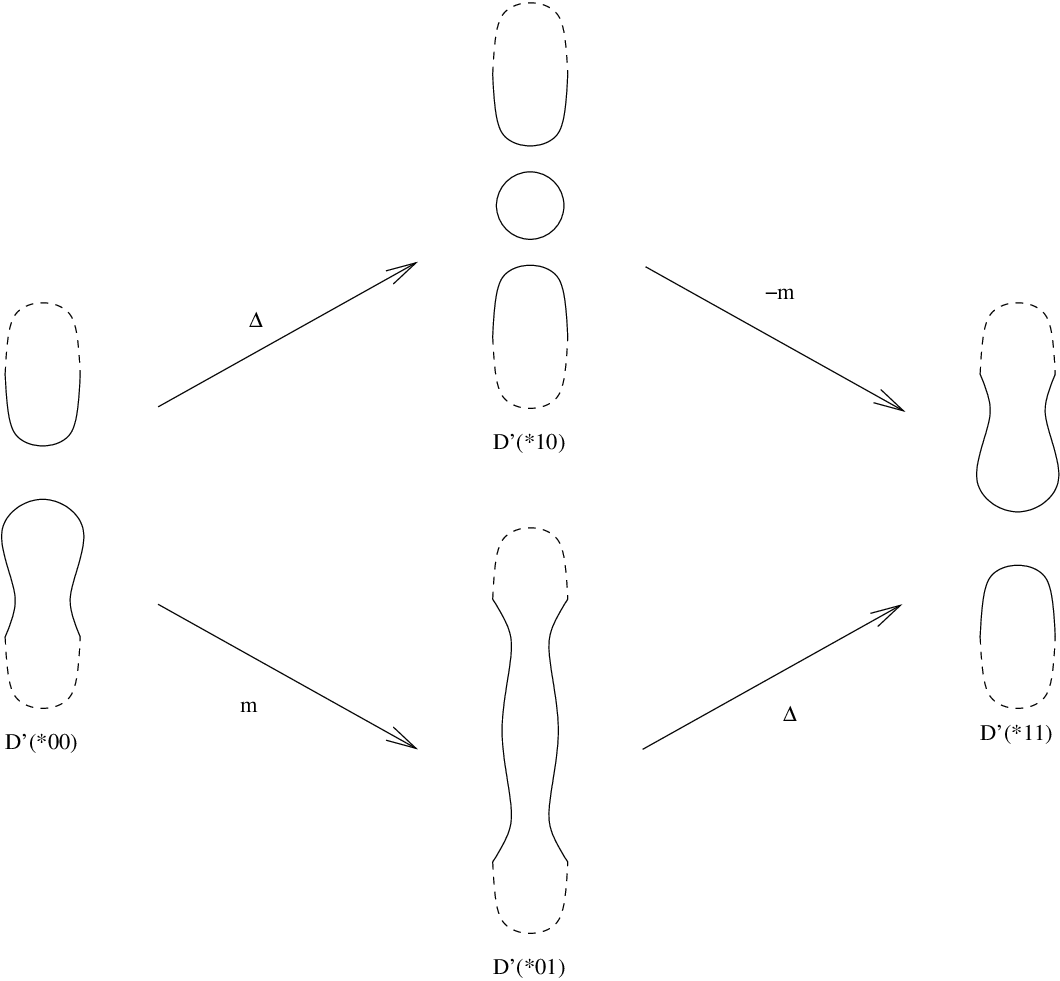} \par\end{centering}

\caption{$\bcn(D')$ for the case when $\partial'_{00\rightarrow11}=-\Delta$}
\end{figure}

Since $X_{1}^{r}=\{ z+\alpha(z)|z\in\bcn(D'(*01))[-1]\{-1\}^{r}\}$,
we can choose the basis: \[
\{(1\otimes x\otimes1+x\otimes1\otimes1,1),(x\otimes x\otimes1,x)\}\]
where the first component is an element of $\bcn(D'(*10))^{r}$ and
the second component lies on $\bcn(D'(*01))^{r}$.

Since $X_{2}^{r}=\{ z+\partial'y|z\in\bcn(D'(*00))^{r},y\in\bcn(D'(*00))^{r-1}\}$,
a basis for the $z$ part (which lies in $\bcn(D'(*00))^{r}$) is
$\{1\otimes1,x\otimes1,1\otimes x,x\otimes x\}$, and for the $\partial'(y)$
part is

\begin{eqnarray*}
(1\otimes1\otimes x+1\otimes x\otimes1 & , & 1)\\
(x\otimes1\otimes x+x\otimes x\otimes1 & , & x)\\
(1\otimes x\otimes x & , & x)\\
(x\otimes x\otimes x & , & 0)\end{eqnarray*}
which lies on $\bcn(D'(*10))^{r}\oplus\bcn(D'(*01))^{r}$.

Now, $X_{3}^{r}=\{ z+y\otimes 1|z\in\bcn(D'(*11))[-2]\{-2\}\}^{r},y\in\bcn(D'(*11))[-2]\{-2\}\}^{r+1}$
has as basis $\{1\otimes1,x\otimes1,1\otimes x,x\otimes x\}$ on $\bcn(D'(*11))[-2]\{-2\}^{r}$
and

\[
\{(1\otimes1\otimes1,0),(x\otimes1\otimes1,0),(1\otimes x\otimes1,0),(x\otimes x\otimes1,0)\}\]
on $\bcn(D'(*10))^{r}\oplus\bcn(D'(*01))^{r}$.

Thus, the bases for $\bcn(D'(*00))^{r}$ and $\bcn(D'(*11))^{r}$
are the distinguished bases. And for $\bcn(D'(*10))^{r}\oplus\bcn(D'(*01))^{r}$,
the change of basis matrix between the ordered basis coming from $X_{1}^{r}$,
$X_{2}^{r}$, and $X_{3}^{r}$ (in the order presented) and the distinguished
basis for $\bcn(D')^{r}$ is:

\begin{equation}
\left(\begin{tabular}{cccccccccc}
 0  &  0  &  0  &  0  &  0  &  0  &  1  &  0  &  0  &  0 \\
1  &  0  &  0  &  0  &  0  &  0  &  0  &  1  &  0  &  0 \\
1  &  0  &  1  &  0  &  0  &  0  &  0  &  0  &  1  &  0 \\
0  &  1  &  0  &  1  &  0  &  0  &  0  &  0  &  0  &  1 \\
0  &  0  &  1  &  0  &  0  &  0  &  0  &  0  &  0  &  0 \\
0  &  0  &  0  &  1  &  0  &  0  &  0  &  0  &  0  &  0 \\
0  &  0  &  0  &  0  &  1  &  0  &  0  &  0  &  0  &  0 \\
0  &  0  &  0  &  0  &  0  &  1  &  0  &  0  &  0  &  0 \\
1  &  0  &  1  &  0  &  0  &  0  &  0  &  0  &  0  &  0 \\
0  &  1  &  0  &  1  &  1  &  0  &  0  &  0  &  0  &  0 \end{tabular}\right)\end{equation}
which also has determinant $\pm1$.

Therefore, the bases for the cochain groups coming from the decomposition
of the complex as the subcomplexes $X_{1}$, $X_{2}$, and $X_{3}$
are equivalent to the distinguished bases in both cases.

The subcomplex $X_{2}=\bcn(D'(*00))\bigoplus\partial'(\bcn(D'(*00)))$
is acyclic. (see \cite{Khovanov1999}) Let $a\in X_{2}^{r}$, then
$a=z+\partial'^{r-1}(y)$ where $z\in\bcn(D'(*00))^{r}$ and $y\in\bcn(D'(*00))^{r-1}$,
hence \begin{eqnarray*}
\partial'^{r}(a) & = & \partial'^{r}(z)+\partial'^{r}(\partial'^{r-1}(y))\\
 & = & \Delta(z)+\partial'_{00\rightarrow01}(z)+\partial'_{00\rightarrow00}(z)\end{eqnarray*}

Let $b'^{r}$ be a basis for $\Img\partial'^{r}\vert_{X_{2}^{r}}$.
Note that since $\Delta$ has trivial kernel, a pullback for this
basis $\tilde{b}'^{r}$ is the distinguished basis $c'(00)^{r}$ for
$\bcn(D'(*00))^{r}$.

Hence, $[b'^{r-1}\medspace\tilde{b}'^{r}/\mathcal{X}_{2}^{r}]$ is
the determinant of the change of basis matrix between the ordered
basis for $b'^{r-1}\medspace\tilde{b}'^{r}$ as described and $\mathcal{X}_{2}^{r}$,
the previously described basis for $X_{2}^{r}$. These bases are the
same since $b'^{r-1}$ is the basis we had for $\partial'^{r-1}\bcn(D'(*00))^{r-1}$
and $\tilde{b}'^{r}$ is the basis we had for $\bcn(D'(*00))^{r}$.
Hence, $[b'^{r-1}\medspace\tilde{b}'^{r}/\mathcal{X}_{2}^{r}]=1$.

Consider the acyclic (see \cite{Khovanov1999}) subcomplex $X_{3}=\{ z+y\otimes 1|z,y\in\bcn(D'(*11))\}$.
Let $a\in X_{3}^{r}$, then $a=z+y\otimes1$ where $z\in\bcn(D'(*11))^{r},y\in\bcn(D'(*11))^{r+1}$and,
hence, \begin{eqnarray*}
\partial'{}^{r}(a) & = & \partial'{}^{r}(z)+\partial'{}^{r}(y\otimes1)\\
 & = & \partial'{}_{11\rightarrow11}^{r}(z)-m(y,1)+\partial'{}_{01\rightarrow01}^{r}(y\otimes1)\\
 & = & \partial'{}_{11\rightarrow11}^{r}(z)-y+\partial'{}_{01\rightarrow01}^{r}(y\otimes1)\end{eqnarray*}

Let $b'^{r}$ be a basis for $\Img\partial'^{r}\vert_{X_{3}^{r}}$.
Note that since $m(*,1)$ generates all of $\bcn(D'(*11))^{r+1}$,
\[
\Img\partial'{}^{r}|_{X_{3}^{r}}={\{(-y,\partial'{}_{01\rightarrow01}^{r}(y\otimes1))|y\in bcn(D'(*11))^{r+1}\}}\]
 hence, a basis is given by: \[
b'^{r}=(-c'(*11)^{r+1},{b'}_{01}^{r})\]
 where $c'(*11)$ is the distinguished basis for $\bcn(D'(*11))^{r}$
and ${b'}_{01}^{r}$ is a basis for $\Img\partial'{}_{01\rightarrow01}^{r}$.
A pullback for this basis is: \[
\tilde{b}'^{r}=(0,c'(*11)^{r+1}\otimes1)\]
Hence, $[b'^{r-1}\medspace\tilde{b}'^{r}/\mathcal{X}_{3}^{r}]$ is
the determinant of the change of basis matrix between the ordered
basis for $b'^{r-1}\medspace\tilde{b}'^{r}$ as described, and the
basis $\mathcal{X}_{3}^{r}$, described previously for $X_{3}^{r}$:\begin{equation*}
\left(
\begin{tabular}{cc}
$-Id$ & $0$ \\
$b'^{r-1}_{01}$ & $Id$
\end{tabular}
\right)
\end{equation*}which has determinant $\pm1$.

Moreover: \begin{eqnarray*}
\isom: & \bcn(D)[-1]\{-1\} & \longrightarrow\; X_{1}\\
 & z & \longmapsto\;(-1)^{r}(z+\alpha(z))\end{eqnarray*}
 induces an isomorphism between $\hm(D)$ and $\hm(D')$. (see \cite{Khovanov1999})
With the prescribed bases, the matrix representation of the cochain
map is the identity matrix. Hence by \eqref{torsion-id-quasi}, the
Reidemeister torsion of the quasi-isomporhism is $1$.

\subsection{Reidemeister III Move}

\begin{figure}[h]

\begin{centering}\includegraphics[scale=0.6]{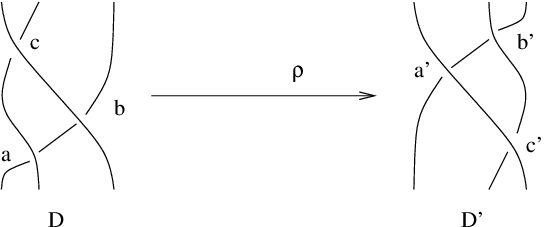} \par\end{centering}

\caption{Reidemeister III invariance}
\end{figure}

\begin{figure}[h]
\begin{centering}\includegraphics[scale=0.3]{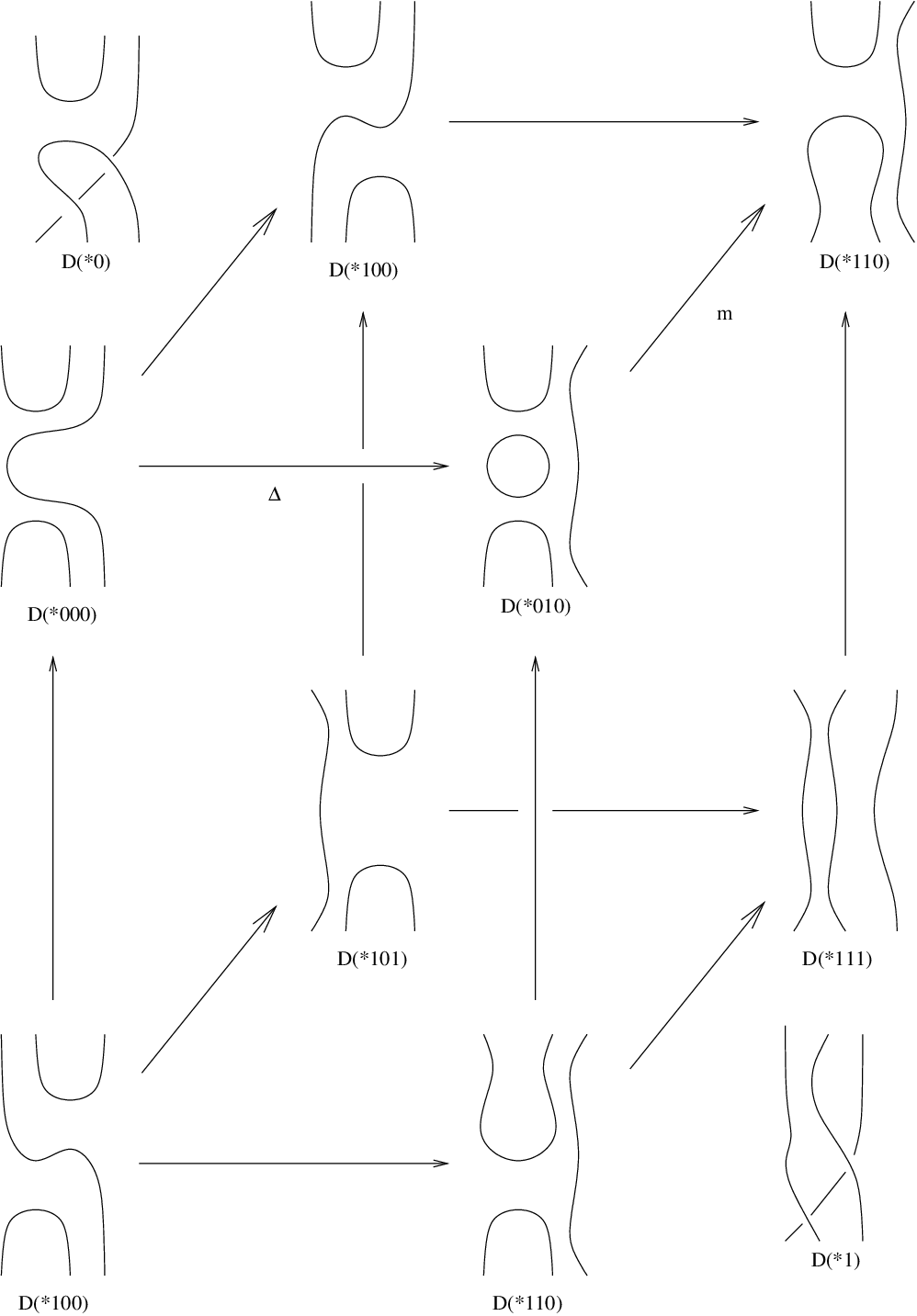} \par\end{centering}

\caption{Complex for the left side $\bcn(D)$\label{RIII-left}}
\end{figure}

\begin{figure}[h]

\begin{centering}\includegraphics[scale=0.3]{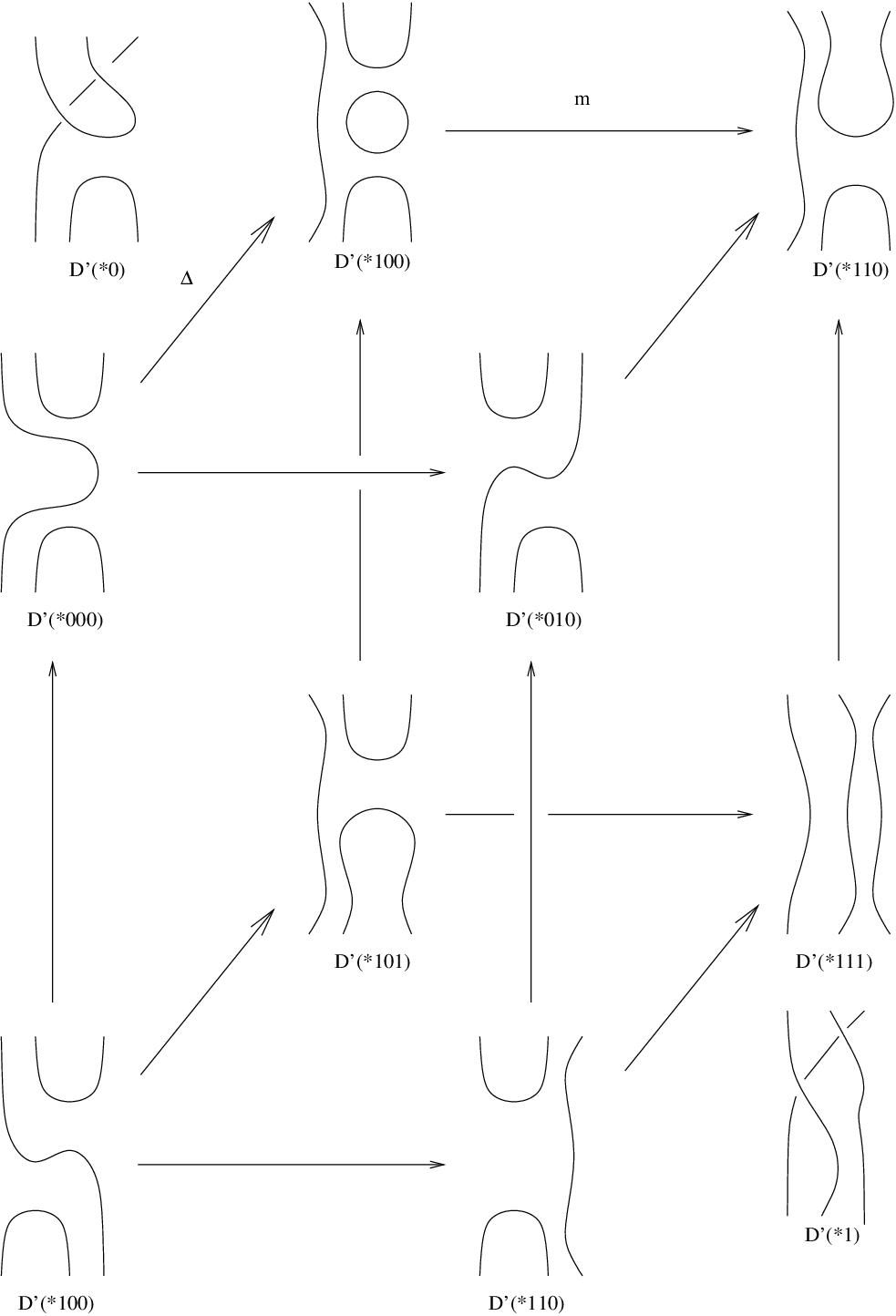} \par\end{centering}

\caption{Complex for the right side $\bcn(D')$\label{RIII-right}}
\end{figure}

Again, $a,b,c$ and $a',b',c'$ are the last three elements in $\indI$
and $\indI'$ and the others are in the same order.

Define $\alpha,\beta,\alpha',\beta'$ as \begin{eqnarray*}
\alpha: & \bcn(D(*110))[-2]\{-2\} & \longrightarrow\;\bcn(D(*010))[-1]\{-1\}\approx\bcn(D(*110))[-1]\{-1\}\otimes\A\\
 & z & \longmapsto\; z\otimes 1\\
\beta: & \bcn(D(*100))[-1]\{-1\} & \longrightarrow\;\bcn(D(*010))[-1]\{-1\}\\
 & z & \longmapsto\;\alpha\partial_{100\rightarrow110}(z)\\
\alpha': & \bcn(D'(*110))[-2]\{-2\} & \longrightarrow\;\bcn(D'(*100))[-1]\{-1\}\approx\bcn(D'(*110))[-1]\{-1\}\otimes\A\\
 & z & \longmapsto\; z\otimes 1\\
\beta': & \bcn(D'(*010))[-1]\{-1\} & \longrightarrow\;\bcn(D'(*100))[-1]\{-1\}\\
 & z & \longmapsto\;-\alpha'\partial'_{010\rightarrow110}(z)\textrm{ .}\end{eqnarray*}
$\bcn(D)$ and $\bcn(D')$ can be decomposed into their subcomplexes
as below: \begin{eqnarray*}
\bcn(D) & = & X_{1}\oplus X_{2}\oplus X_{3}\\
X_{1} & = & \{ x+\beta(x)+y|x\in\bcn(D(*100))[-1]\{-1\},y\in\bcn(D(*1))[-1]\{-1\}\}\\
X_{2} & = & \{ x+\partial y|x,y\in\bcn(D(*000))\}\\
X_{3} & = & \{\alpha(x)+\partial\alpha(y)|x,y\in\bcn(D(*110))[-2]\{-2\}\}\\
\bcn(D') & = & Y_{1}\oplus Y_{2}\oplus Y_{3}\\
Y_{1} & = & \{ x+\beta'(x)+y|x\in\bcn(D'(*010))[-1]\{-1\},y\in\bcn(D'(*1))[-1]\{-1\}\}\\
Y_{2} & = & \{ x+\partial'y|x,y\in\bcn(D'(*000))\}\\
Y_{3} & = & \{\alpha'(x)+\partial'\alpha'(y)|x,y\in\bcn(D'(*110))[-2]\{-2\}\}\end{eqnarray*}

On the cube on Figure (\ref{RIII-left}), the complex is decomposed
as the three subcomplexes $X_{1}$, $X_{2}$ and $X_{3}$. In the
bottom of the cube, the decomposition does not change the bases for
the cochain groups. Let's look at the top of the cube and see how
this decomposition affects torsion. Note that the upper left cup is
unchanged through the diagram. Hence, we can reduce the diagram since
we are going to study locally the change of basis. Moreover, consider
first the case when $\partial_{100\rightarrow110}=-m$.

\begin{figure}[h]

\begin{centering}\includegraphics[scale=0.6]{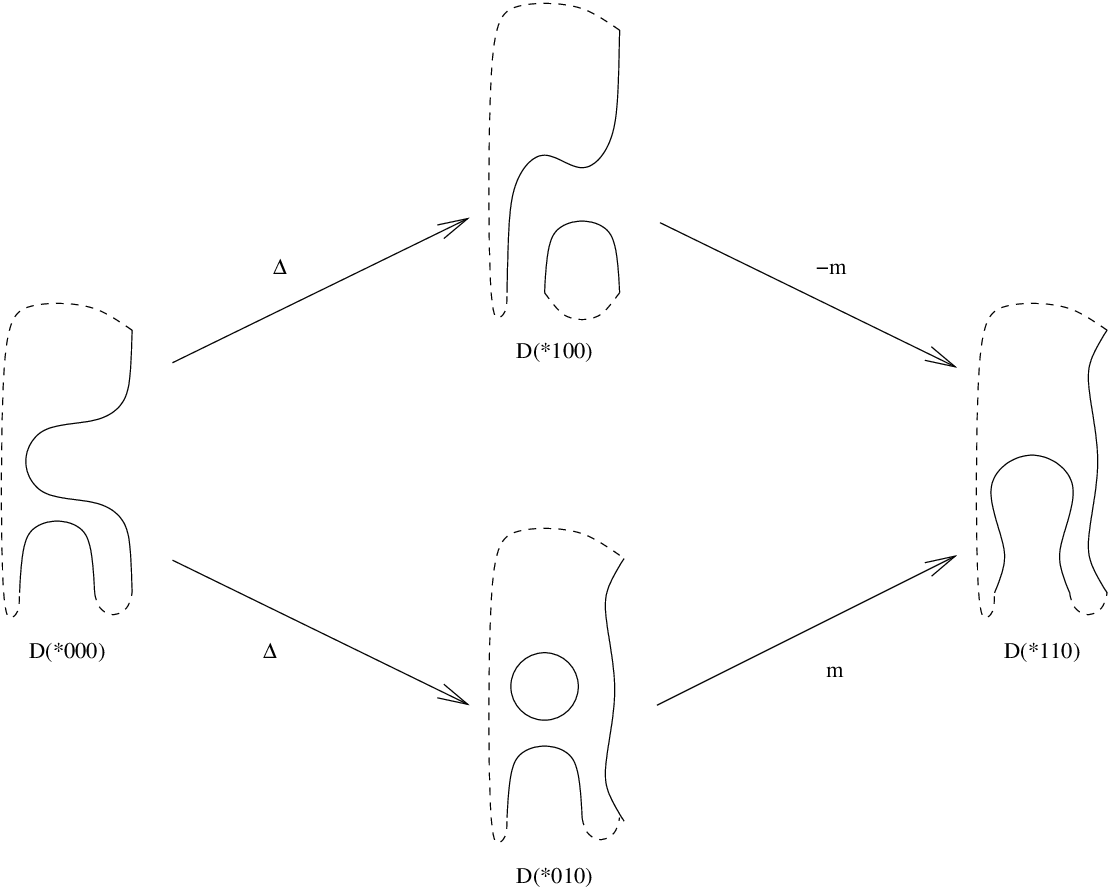} \par\end{centering}

\caption{$\bcn(D)$ for the case when $\partial_{100\rightarrow110}=-m$}
\end{figure}

Since $X_{1}^{r}=\{ x+\beta(x):x\in\bcn(D(*100))[-1]\{-1\}^{r}\}$
then we get the basis: \begin{eqnarray*}
(1\otimes1 & , & -1\otimes1),\\
(x\otimes1 & , & -x\otimes1),\\
(1\otimes x & , & -x\otimes1),\\
(x\otimes x & , & 0)\end{eqnarray*}
 which lies in $\bcn(D(*100)^{r}\oplus\bcn(D(*010))^{r}$.

From $X_{2}^{r}=\{ x+\partial y|x\in\bcn(D(*000))^{r},y\in\bcn(D(*000))^{r-1}\}$
we get the basis $\{1,x\}$ for $\bcn(D(*000))^{r}$ and for the part
that lies in $\bcn(D(*100)^{r}\oplus\bcn(D(*010))^{r}$, we get: \[
\{(1\otimes x+x\otimes1,1\otimes x+x\otimes1),(x\otimes x,x\otimes x)\}\]
 Finally, since $X_{3}^{r}=\{\alpha(x)+\partial\alpha(y)|x\in\bcn(D(*110))[-2]\{-2\}^{r+1},y\in\bcn(D(*110))[-2]\{-2\}^{r}\}$,
we get the basis $\{1,x\}$ for $\bcn(D(*110))^{r}$ and for the part
that lies in $\bcn(D(*100)^{r}\oplus\bcn(D(*010))^{r}$, we get: \[
\{(0,1\otimes1),(0,x\otimes1)\}\]
 Therefore, the change of basis matrix for $\bcn(D(*100)^{i}\oplus\bcn(D(*010))^{i}$
between the bases described for $X_{1}^{r}$,$X_{2}^{r}$, and $X_{3}^{r}$
and the distinguished basis is: \begin{equation}
\left(\begin{tabular}{cccccccc}
 1  &  0  &  0  &  0  &  0  &  0  &  0  &  0 \\
0  &  1  &  0  &  0  &  1  &  0  &  0  &  0 \\
0  &  0  &  1  &  0  &  1  &  0  &  0  &  0 \\
0  &  0  &  0  &  1  &  0  &  1  &  0  &  0 \\
-1  &  0  &  0  &  0  &  0  &  0  &  1  &  0 \\
0  &  -1  &  -1  &  0  &  1  &  0  &  0  &  1 \\
0  &  0  &  0  &  0  &  1  &  0  &  0  &  0 \\
0  &  0  &  0  &  0  &  0  &  1  &  0  &  0 \end{tabular}\right)\end{equation}
 which has determinant $\pm1$.

For the case when $\partial_{100\rightarrow110}=-\Delta$, the complex
is shown. %
\begin{figure}[h]

\begin{centering}\includegraphics[scale=0.6]{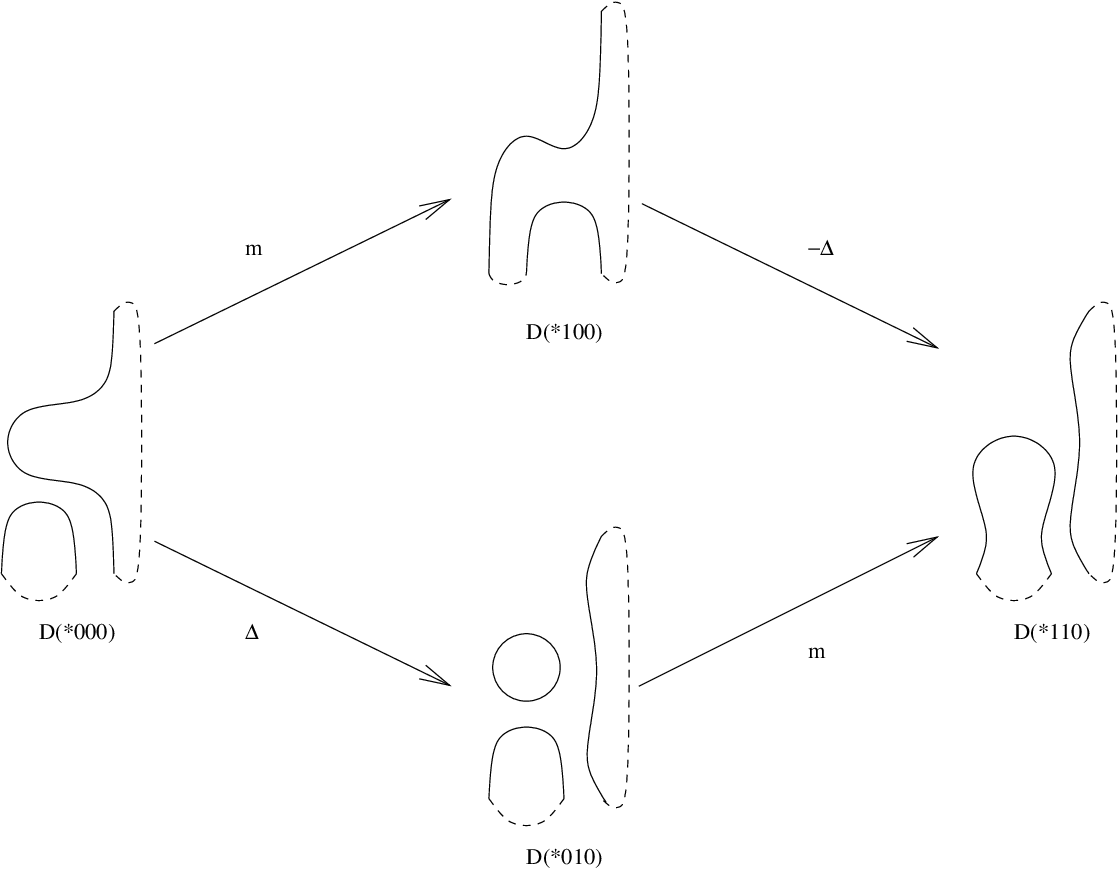} \par\end{centering}

\caption{$\bcn(D)$ for the case when $\partial_{100\rightarrow110}=-\Delta$}
\end{figure}

Then $X_{1}^{r}$ has basis: \[
\{(1,-1\otimes x\otimes1-x\otimes1\otimes1),(x,-x\otimes x\otimes1)\}\]
 which lies in $\bcn(D(*100)^{r}\oplus\bcn(D(*010))^{r}$.

As before $X_{2}^{r}$, has the distinguished basis for $\bcn(D(*000))^{r}$
and its contribution to $\bcn(D(*100)^{r}\oplus\bcn(D(*010))^{r}$
is:

\begin{eqnarray*}
(1 & , & 1\otimes1\otimes x+x\otimes1\otimes1)\\
(x & , & 1\otimes x\otimes x+x\otimes x\otimes1)\\
(x & , & x\otimes1\otimes x)\\
(0 & , & x\otimes x\otimes x)\end{eqnarray*}
Also as before, $X_{3}^{r}$ yields the distinguished basis for $\bcn(D(*110))^{r}$
and for $\bcn(D(*100)^{r}\oplus\bcn(D(*010))^{r}$ it contributes:
\begin{eqnarray*}
(0 & , & 1\otimes1\otimes1)\\
(0 & , & x\otimes1\otimes1)\\
(0 & , & 1\otimes x\otimes1)\\
(0 & , & x\otimes x\otimes1)\end{eqnarray*}
Therefore, the change of basis matrix for $\bcn(D(*100)^{r}\oplus\bcn(D(*010))^{r}$
between the bases described for $X_{1}^{r}$,$X_{2}^{r}$, and $X_{3}^{r}$
and the distinguished basis is: \begin{equation}
\left(\begin{tabular}{cccccccccc}
 1  &  0  &  1  &  0  &  0  &  0  &  0  &  0  &  0  &  0 \\
0  &  1  &  0  &  1  &  1  &  0  &  0  &  0  &  0  &  0 \\
0  &  0  &  0  &  0  &  0  &  0  &  1  &  0  &  0  &  0 \\
-1  &  0  &  1  &  0  &  0  &  0  &  0  &  1  &  0  &  0 \\
-1  &  0  &  0  &  0  &  0  &  0  &  0  &  0  &  1  &  0 \\
0  &  -1  &  0  &  1  &  0  &  0  &  0  &  0  &  0  &  1 \\
0  &  0  &  1  &  0  &  0  &  0  &  0  &  0  &  0  &  0 \\
0  &  0  &  0  &  0  &  1  &  0  &  0  &  0  &  0  &  0 \\
0  &  0  &  0  &  1  &  0  &  0  &  0  &  0  &  0  &  0 \\
0  &  0  &  0  &  0  &  0  &  1  &  0  &  0  &  0  &  0 \end{tabular}\right)\end{equation}
 which has determinant $\pm1$.

In both cases, the bases for the cochain groups $\bcn(D(*000))^{r}$
and $\bcn(D(*110))^{r}$ are the distinguished bases. For the cochain
groups $\bcn(D(*100)^{r}$ and $\bcn(D(*010))^{r}$, the change of
bases matrices have determinant $\pm1$. Hence, the bases for the
cochain groups coming from the decomposition of the complex as the
subcomplexes $X_{1}$, $X_{2}$, and $X_{3}$ are equivalent to the
distinguished bases.

On the right side, Figure (\ref{RIII-right}), the cube is also decomposed
as three subcomplexes $X_{1}$, $X_{2}$, and $X_{3}$. Note that
the decomposition does not change the basis used in the bottom of
the cube. Let's look at the top of the cube after reducing the bottom
right cup, which is constant throughout.

\begin{figure}[h]

\begin{centering}\includegraphics[scale=0.6]{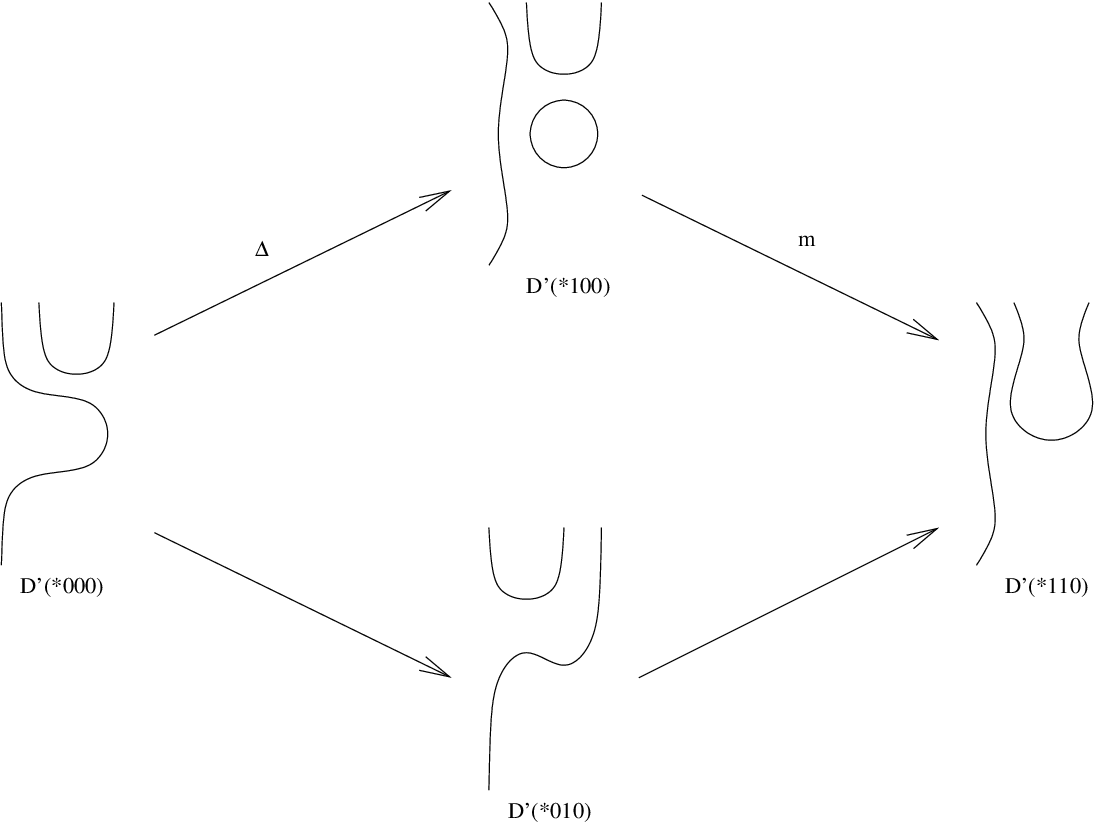} \par\end{centering}

\caption{Reduced diagram for $\bcn(D')$}
\end{figure}

After rotating each diagram by $180^{\circ}$, and interchanging the
placement of the middle cochain groups in the direct sum, we obtain
the same complex as the one on the top of the left cube. The reader
should look at the way the complex is decomposed and note the change
of basis done for the the cube on Figure (\ref{RIII-left}) is similar
to the one done here.

Note that the acyclic (see \cite{Khovanov1999}) subcomplexes: \begin{eqnarray*}
X_{2} & = & \bcn(D(*000))\bigoplus\partial(\bcn(D(*000)))\\
X'_{2} & = & \bcn(D'(*000))\bigoplus\partial'(\bcn(D'(*000)))\end{eqnarray*}
 are of the form $D\bigoplus\partial(D)$ and (since $\Delta$ has
trivial kernel) they behave as the case for $X_{2}$ on the Reidemeister
II move.

Furthermore, the acyclic (see \cite{Khovanov1999}) subcomplexes:
\begin{eqnarray*}
X_{3} & = & \alpha(\bcn(D(*110)))\bigoplus\partial(\alpha(\bcn(D(*110))))\\
X'_{3} & = & \alpha'(\bcn(D'(*110)))\bigoplus\partial'(\alpha'(\bcn(D'(*110))))\end{eqnarray*}
 behave exactly like the case for $X_{3}$ in the Reidemeister II
move. (since $m(*,1)$ is an onto map)

Furthermore: $\bcn(D(*100))[-1]\{-1\}$ and $\bcn(D'(*010))[-1]\{-1\}$,
$\bcn(D(*1))[-1]\{-1\}$ and $\bcn(D'(*1))[-1]\{-1\}$ are naturally
isomorphic, and $X_{1}$ is isomorphic to $Y_{1}$ via \[
\isom:x+\beta(x)+y\longmapsto x+\beta'(x)+y\textrm{ .}\]
 and with the bases described the matrix representation of the quasi-isomorphism
is the identity map for all the cochain groups and thus by \eqref{torsion-id-quasi}
its Reidemeister torsion is $1$.

Therefore, for all the Reidemeister moves, the quasi-isomorphisms
have Reidemeister torsion $1$. If two diagrams represent the same
knot, then one can go from one to the other with a series of Reidemeister
moves. The volume form is preserved by each move, giving an invariant
volume form on the Khovanov invariant. The Khovanov chain complex
decomposes into subcomplexes by polynomial degree. Since the acyclic
subcomplexes used to demonstrate invariance have Reidemeister torsion
$1$, for acyclic subcomplexes corresponding to these polynomial degrees,
Reidemeister torsion will be a number (there is no homology present),
which will also be invariant for knots and links.

\subsection{Basis for homology groups}

Let's look at a special case of our construction by making the following
choices for bases for the cohomology groups. Consider the cochain
groups to have the same generators as explained before but being over
$\mathbb{Z}$ instead of a field, then the cohomology groups (by the
Fundamental Theorem of Abelian Groups) are of the form: \[
H(\mathbb{Z})\cong\mathbb{Z}^{n}\oplus\mathbb{Z}_{p_{i}}^{n_{i}}\]
 where the $p_{i}$'s are prime numbers. Then, the quotient of $H(\mathbb{Z})$
by its torsion subgroup is free and it has the same rank (over $\mathbb{Z}$)
as $H(\mathbb{F})$ (over $\mathbb{F}$). Over $\mathbb{Z}$ we can
pick a basis for the free part of the homology groups and any two
bases are related by a transition matrix in $Gl(n,\mathbb{Z})$. Since
Khovanov homology is invariant over $\mathbb{Z}$, using this basis
for the homology groups we obtain a numerical invariant for knots
and links.

Computationally we can obtain such a basis by using the Smith normal
form. Consider the short exact sequence: \begin{equation}
0\rightarrow B^{k-1}\xrightarrow{X^{k}}Z^{k}\xrightarrow{\pi}H^{k}\rightarrow0\end{equation}
 where $X_{k}$ is the matrix representation of the injection map
with respect to a basis for $B^{k-1}$ and a $\mathbb{Z}$-basis for
$Z^{k}=\ker\partial^{k}$. Now consider the Smith normal form of $X^{k}$:
\begin{equation}
D=snf(X^{k})=\left(\begin{matrix}0 & 0 & \dotsb & 0\\
\vdots & \vdots & \ddots & \vdots\\
0 & 0 & \dotsb & 0\\
d_{1} & 0 & \dotsc & 0\\
0 & d_{2} & \dotsc & 0\\
\vdots & \dotsb & \ddots & \vdots\\
0 & 0 & \dotsc & d_{l}\end{matrix}\right)\end{equation}
 where $d_{i}=1$ or $d_{i}$ is an invariant factor of the matrix
$X_{k}$. In particular, there are unitary matrices $V$ and $U$
such that $U\cdot X^{k}\cdot V=D$. In other words, $U$ and $V$
are change of basis matrices for $B^{k-1}$ and $Z^{k}$, respectively,
that make the injection map have matrix representation $D$. Zero
rows of $D$ will give a basis for the cohomology group after applying
the change of basis matrix $U$.

Since we are only interested in a basis for the cohomology groups,
we can restrict our study to the case where $U\cdot X^{k}=D$. 

\begin{theorem}
Let $U$ and $U^{'}$ be unitary matrices such that $U\cdot X^{k}=D=U^{'}\cdot X^{k}$
where $D$ and $X^{k}$ are as specified above. Then the basis for
the cohomology groups obtained from the zero rows of $D$ after applying
the change of bases matrices $U$ and $U^{'}$ respectively are equivalent.
\end{theorem}
\begin{proof}
By our hypothesis:\begin{equation}
\begin{split}U\cdot X^{k}= & D=U^{'}\cdot X^{k}\qquad\Leftrightarrow\\
U^{-1}\cdot D= & X^{k}=U^{'-1}\cdot D\quad\Leftrightarrow\\
U\cdot U^{'-1}\cdot D= & D\end{split}
\end{equation}
 From the structure of $D$, and column operations we can get that:
\begin{equation}
U\cdot U^{'-1}=\left(\begin{matrix}M & 0\\
* & Id\end{matrix}\right)\thicksim\left(\begin{matrix}M & 0\\
0 & Id\end{matrix}\right)\end{equation}
 Moreover, since $U$ and $U^{'}$ are unitary: \begin{equation}
\begin{split}1=1\ \cdot1=(\det U)\cdot(\det U^{'-1}) & =\det(U\cdot U^{'-1})\\
 & =\det(M)\cdot\det(Id)\\
 & =\det(M)\end{split}
\label{hom-base-change}\end{equation}
 Since any two $\mathbb{Z}$-basis for $Z^{k}$ have a change of basis
matrix in $Gl(n,\mathbb{Z})$ and by \eqref{hom-base-change}, any
two basis that we picked for our cohomology groups also have a change
of basis $M\in Gl(n,\mathbb{Z})$, the bases for cohomology groups
are equivalent, and hence do not change the torsion computation. 
\end{proof}

\section{Examples}

\label{sec:Examples}

In 2004 Shumakovitch wrote KhoHo, a program for computing and studying
Khovanov homology (see \cite{Shumakovitch2004}). It is run using
the PARI/GP calculator. For this reason, the routines for computing
Reidemeister torsion for the Khovanov cochain complex were also written
using PARI/GP code. The Khovanov cochain complex can be decompose
by the grading in the vector space $V$, and in this way from each
of these subcomplexes one obtains the corresponding coefficient for
the term of that particular degree in the Jones polynomial. From KhoHo,
by inputing a planar diagram representation for a given diagram, one
can easily obtain the coboundary operators for each of the subcomplexes
that form the Khovanov cochain complex.

\subsection{A detailed example}

The Hopf link contains generators in homological degrees $-2$, $-1$,
and $0$, and it is decomposed by polynomial degree into four subcomplexes
of degrees $0$, $-2$, $-4$, and $-6$. The output of the RTorsion
routine returns a matrix. 

\begin{verbatim}
 ? RTorsion(hopf)
%1 = 

[1 1 1 1 "0"]

[1 -1 1 1 "-2"]

[1 -1 1 1 "-4"]

[-1 1 1 1 "-6"]

\end{verbatim} Each row correspond to a different subcomplex. In the last column
of each row, the number inside the quotation marks indicates the polynomial
degree for that particular subcomplex. The number in the previous
column is the Reidemeister torsion for the subcomplex whose degree
is specified in the last column. The rest of the entries in the matrix
give the contribution to the torsion coming from each of the cochain
groups contained in that particular subcomplex. For example, the subcomplex
of polynomial degree $-4$ is 

$$0 \rightarrow C^{-2}_{-4} \xrightarrow{\partial^{-2}} C^{-1}_{-4} \xrightarrow{\partial^{-1}} C^0_{-4} \rightarrow 0$$
and $[h^{-2}\tilde{b}^{-2}/c^{-2}]=1$, $[b^{-2}\tilde{b}^{-1}/c^{-1}]=-1$
(there is no homology here) and $[b^{-1}h^{0}/c^{0}]=1$. Therefore
$\tau_{-4}=\biggl\lvert\frac{[b^{-2}\tilde{b}^{-1}/c^{-1}]}{[h^{-2}\tilde{b}^{-2}/c^{-2}]\cdot[b^{-1}h^{0}/c^{0}]}\biggr\rvert=\biggl\lvert\frac{-1}{1\cdot1}\biggr\rvert=1$.
Note that these are the entries for the row corresponding to this
subcomplex.

\subsection{Knots}

\label{sec:R-torsion-knots}The names used for the knots in this table
follow Rolfsen's table of knots \cite{Rolfsen2003}. The R-torsion
for each knot is given as a table that contains the contribution to
torsion from each cochain group (i.e. the determinant of the change
of basis matrix between the basis obtained from the boundary operators
and $\mathbb{Z}$-homology and the distinguished basis), and at the
next to last column it contains the R-torsion for the subcomplex (i.e.
the alternating product of the determinants of the change of basis
matrices coming from each of the cochain groups) whose degree is specified
in the last column.

\begin{longtable}{|c|r|}
\hline
\multicolumn{2}{|c|}{{\tablename} \thetable{}: R-torsion for knots with up to $7$ crossings}
\addcontentsline{lot}{table}{\thetable{} \quad R-torsion for knots with up to $7$ crossings}\\
\hline
\hline
knot&
Reidemeister Torsion\tabularnewline
\hline
\endfirsthead
\hline

\multicolumn{2}{|c|}{{\tablename} \thetable{} -- continued}\\
\hline
\hline
\endhead

knot3[1]
&[1 1 1 1 1 "-1"] \\*

&[1 1 1 1 1 "-3"] \\*

&[1 -1 1 1 1 "-5"] \\*

&[1 -2 1 1 (1 /2) "-7"] \\*

&[1 1 1 1 1 "-9"]
\tabularnewline \hline
knot4[1]
&[1 1 1 1 1 1 "5"] \\*

&[1 1 1 -1 2 (1 /2) "3"] \\*

&[1 -1 1 1 1 1 "1"] \\*

&[1 -1 -1 -1 1 1 "-1"] \\*

&[-1 -2 1 1 1 2 "-3"] \\*

&[1 1 1 1 1 1 "-5"]
\tabularnewline \hline
knot5[1]
&[1 1 1 1 1 1 1 "-3"] \\*

&[1 1 -1 -1 1 1 1 "-5"] \\*

&[1 1 -1 1 1 1 1 "-7"] \\*

&[-1 1 1 2 1 1 (1 /2) "-9"] \\*

&[1 -1 1 1 1 1 1 "-11"] \\*

&[-1 2 1 1 1 1 (1 /2) "-13"] \\*

&[1 1 1 1 1 1 1 "-15"]
\tabularnewline \hline
knot5[2]
&[1 1 1 1 1 1 1 "-1"] \\*

&[1 1 1 1 1 -1 1 "-3"] \\*

&[1 1 1 1 2 -1 2 "-5"] \\*

&[1 -1 -1 2 1 -1 (1 /2) "-7"] \\*

&[-1 1 -1 1 1 1 1 "-9"] \\*

&[1 -2 1 1 1 1 (1 /2) "-11"] \\*

&[1 1 1 1 1 1 1 "-13"]
\tabularnewline \hline
knot6[1]
&[1 1 1 1 1 1 1 1 "5"] \\*

&[1 1 1 1 1 -1 -2 (1 /2) "3"] \\*

&[1 1 1 1 1 1 -1 1 "1"] \\*

&[1 1 -1 -1 -2 1 1 (1 /2) "-1"] \\*

&[1 -1 -1 -2 1 1 -1 2 "-3"] \\*

&[-1 1 1 1 1 -1 1 1 "-5"] \\*

&[1 2 -1 -1 1 1 1 2 "-7"] \\*

&[1 1 1 1 1 1 1 1 "-9"]
\tabularnewline \hline
knot6[2]
&[1 1 1 1 1 1 1 1 "3"] \\*

&[1 1 1 1 1 1 -2 (1 /2) "1"] \\*

&[1 -1 -1 -1 -1 1 -1 1 "-1"] \\*

&[1 -1 -1 1 -2 -1 1 (1 /2) "-3"] \\*

&[-1 -1 1 -2 1 1 1 2 "-5"] \\*

&[-1 -1 -2 1 1 1 1 (1 /2) "-7"] \\*

&[-1 2 1 1 1 1 1 2 "-9"] \\*

&[1 1 1 1 1 1 1 1 "-11"]
\tabularnewline \hline
knot6[3]
&[1 1 1 1 1 1 1 1 "7"] \\*

&[1 1 1 1 1 1 -2 2 "5"] \\*

&[1 1 1 1 1 2 1 (1 /2) "3"] \\*

&[1 1 1 -1 -2 1 -1 2 "1"] \\*

&[1 1 1 2 -1 -1 1 (1 /2) "-1"] \\*

&[1 -1 -2 -1 1 1 1 2 "-3"] \\*

&[1 -2 1 1 1 1 1 (1 /2) "-5"] \\*

&[1 1 1 1 1 1 1 1 "-7"]
\tabularnewline \hline
knot7[1]
&[1 1 1 1 1 1 1 1 1 "-5"] \\*

&[1 1 -1 -1 -1 -1 1 1 1 "-7"] \\*

&[1 -1 -1 -1 1 -1 1 1 1 "-9"] \\*

&[1 -1 -1 -1 1 2 1 1 (1 /2) "-11"] \\*

&[-1 1 -1 1 1 1 1 1 1 "-13"] \\*

&[-1 -1 -1 -2 1 1 1 1 (1 /2) "-15"] \\*

&[-1 -1 1 1 1 1 1 1 1 "-17"] \\*

&[-1 -2 1 1 1 1 1 1 (1 /2) "-19"] \\*

&[1 1 1 1 1 1 1 1 1 "-21"]
\tabularnewline \hline
knot7[2]
&[1 1 1 1 1 1 1 1 1 "-1"] \\*

&[1 1 1 1 1 1 1 1 1 "-3"] \\*

&[1 1 1 1 -1 -1 2 -1 2 "-5"] \\*

&[1 1 1 1 1 -2 1 -1 (1 /2) "-7"] \\*

&[1 1 1 1 2 -1 -1 1 2 "-9"] \\*

&[1 -1 1 2 -1 -1 1 -1 (1 /2) "-11"] \\*

&[-1 -1 -1 1 -1 -1 1 1 1 "-13"] \\*

&[1 -2 -1 -1 1 1 1 1 (1 /2) "-15"] \\*

&[1 1 1 1 1 1 1 1 1 "-17"]
\tabularnewline \hline
knot7[3]
&[1 1 1 1 1 1 1 1 1 "19"] \\*

&[1 1 1 1 1 -1 1 2 2 "17"] \\*

&[1 1 1 -1 -1 1 -1 -1 1 "15"] \\*

&[1 1 -1 -1 1 4 -1 -1 4 "13"] \\*

&[1 1 -1 -1 -2 -1 -1 -1 (1 /2) "11"] \\*

&[-1 1 1 2 -1 -1 1 1 2 "9"] \\*

&[-1 1 -2 -1 1 1 1 1 (1 /2) "7"] \\*

&[1 -1 -1 -1 1 1 1 1 1 "5"] \\*

&[1 1 1 1 1 1 1 1 1 "3"]
\tabularnewline \hline
knot7[4]
&[1 1 1 1 1 1 1 1 1 "17"] \\*

&[1 1 1 1 1 -1 -1 2 2 "15"] \\*

&[1 1 -1 -1 1 1 -1 -1 1 "13"] \\*

&[1 1 -1 1 -1 4 1 1 4 "11"] \\*

&[-1 1 1 1 -2 -1 1 1 (1 /2) "9"] \\*

&[1 1 -1 2 1 1 1 1 2 "7"] \\*

&[1 -1 4 1 1 1 1 1 (1 /4) "5"] \\*

&[1 1 1 1 1 1 1 1 1 "3"] \\*

&[1 1 1 1 1 1 1 1 1 "1"]
\tabularnewline \hline
knot7[5]
&[1 1 1 1 1 1 1 1 1 "-3"] \\*

&[1 1 1 1 1 -1 1 -1 1 "-5"] \\*

&[1 1 1 1 1 -1 2 -1 2 "-7"] \\*

&[1 -1 1 1 -1 -4 1 -1 (1 /4) "-9"] \\*

&[1 -1 1 1 -2 1 1 1 2 "-11"] \\*

&[1 -1 1 4 1 1 1 1 (1 /4) "-13"] \\*

&[1 1 -2 1 1 1 1 1 2 "-15"] \\*

&[1 2 1 1 1 1 1 1 (1 /2) "-17"] \\*

&[1 1 1 1 1 1 1 1 1 "-19"]
\tabularnewline \hline
knot7[6]
&[1 1 1 1 1 1 1 1 1 "3"] \\*

&[1 1 1 1 1 1 1 -2 (1 /2) "1"] \\*

&[1 1 1 -1 -1 -1 2 1 2 "-1"] \\*

&[1 -1 -1 1 1 -2 1 -1 (1 /2) "-3"] \\*

&[1 1 1 -1 4 1 1 1 4 "-5"] \\*

&[1 -1 -1 -4 1 1 1 1 (1 /4) "-7"] \\*

&[-1 1 -2 1 1 1 1 1 2 "-9"] \\*

&[1 -2 1 1 1 1 1 1 (1 /2) "-11"] \\*

&[1 1 1 1 1 1 1 1 1 "-13"]
\tabularnewline \hline
knot7[7]
&[1 1 1 1 1 1 1 1 1 "9"] \\*

&[1 1 1 1 1 1 -1 -2 (1 /2) "7"] \\*

&[1 1 1 1 1 1 -2 -1 2 "5"] \\*

&[1 1 -1 -1 1 -4 -1 -1 (1 /4) "3"] \\*

&[-1 -1 1 -1 4 -1 1 1 4 "1"] \\*

&[-1 -1 -1 -2 -1 1 1 1 (1 /2) "-1"] \\*

&[-1 1 -4 -1 1 1 1 1 4 "-3"] \\*

&[-1 2 1 1 1 1 1 1 (1 /2) "-5"] \\*

&[1 1 1 1 1 1 1 1 1 "-7"]
\tabularnewline
\hline
\end{longtable}

\subsection{Links}

\label{sec:R-torsion-links}The names used for the links in this table
follow the Thistlethwaite link table \cite{Thistlethwaite1991}.
(See the previous section to understand the table.)

\begin{longtable}{|c|r|}
\hline
\multicolumn{2}{|c|}{{\tablename} \thetable{}: R-torsion for links with up to $7$ crossings}
\addcontentsline{lot}{table}{\thetable{} \quad R-torsion for links with up to $7$ crossings}\\
\hline
\hline
link&
Reidemeister Torsion\tabularnewline
\hline
\endfirsthead
\hline

\multicolumn{2}{|c|}{{\tablename} \thetable{} -- continued}\\
\hline
\hline
\endhead
 
link2a[1]
&[1 1 1 1 "0"] \\*

&[1 -1 1 1 "-2"] \\*

&[-1 -1 1 1 "-4"] \\*

&[1 1 1 1 "-6"]
\tabularnewline \hline 
link4a[1]
&[1 1 1 1 1 1 "0"] \\*

&[1 1 1 1 -1 1 "-2"] \\*

&[1 1 -1 2 -1 2 "-4"] \\*

&[1 -1 1 -1 -1 1 "-6"] \\*

&[-1 -1 -1 -1 1 1 "-8"] \\*

&[1 1 1 1 1 1 "-10"]
\tabularnewline \hline 
link5a[1]
&[1 1 1 1 1 1 1 "4"] \\*

&[1 1 1 1 1 -2 (1 /2) "2"] \\*

&[1 1 1 1 1 1 1 "0"] \\*

&[1 -1 1 -1 1 1 1 "-2"] \\*

&[1 -1 -2 1 1 1 2 "-4"] \\*

&[-1 -2 1 1 1 1 (1 /2) "-6"] \\*

&[1 1 1 1 1 1 1 "-8"]
\tabularnewline \hline 
link6a[1]
&[1 1 1 1 1 1 1 1 "4"] \\*

&[1 1 1 1 1 1 2 (1 /2) "2"] \\*

&[1 1 1 -1 1 -2 -1 2 "0"] \\*

&[1 1 -1 -1 1 1 -1 1 "-2"] \\*

&[-1 -1 -1 -4 -1 -1 1 4 "-4"] \\*

&[-1 1 -2 1 1 1 1 (1 /2) "-6"] \\*

&[1 1 1 1 1 1 1 1 "-8"] \\*

&[1 1 1 1 1 1 1 1 "-10"]
\tabularnewline \hline 
link6a[2]
&[1 1 1 1 1 1 1 1 "-2"] \\*

&[1 1 1 1 1 1 -1 1 "-4"] \\*

&[1 1 1 -1 -1 2 -1 2 "-6"] \\*

&[1 1 1 1 -2 -1 -1 (1 /2) "-8"] \\*

&[-1 1 -1 -2 -1 -1 1 2 "-10"] \\*

&[1 1 -2 1 1 1 1 (1 /2) "-12"] \\*

&[-1 -1 1 1 1 1 1 1 "-14"] \\*

&[1 1 1 1 1 1 1 1 "-16"]
\tabularnewline \hline 
link6a[3]
&[1 1 1 1 1 1 1 1 "-4"] \\*

&[1 -1 -1 1 -1 -1 1 1 "-6"] \\*

&[1 -1 -1 1 -1 -1 1 1 "-8"] \\*

&[-1 -1 1 1 2 1 1 (1 /2) "-10"] \\*

&[1 -1 1 -1 1 1 1 1 "-12"] \\*

&[-1 -1 -2 1 1 1 1 (1 /2) "-14"] \\*

&[-1 1 1 1 1 1 1 1 "-16"] \\*

&[1 1 1 1 1 1 1 1 "-18"]
\tabularnewline \hline 
link6a[4]
&[1 1 1 1 1 1 1 1 "7"] \\*

&[1 1 1 1 1 1 -2 2 "5"] \\*

&[1 1 1 1 1 4 -1 (1 /4) "3"] \\*

&[1 -1 -1 -1 1 1 -1 1 "1"] \\*

&[1 1 -1 1 -1 -1 1 1 "-1"] \\*

&[1 -1 -4 -1 1 1 1 4 "-3"] \\*

&[1 -2 1 1 1 1 1 (1 /2) "-5"] \\*

&[1 1 1 1 1 1 1 1 "-7"]
\tabularnewline \hline 
link6a[5]
&[1 1 1 1 1 1 1 1 "-1"] \\*

&[1 1 1 1 1 1 1 1 "-3"] \\*

&[1 1 1 1 -1 -4 -1 4 "-5"] \\*

&[1 1 1 1 2 -1 1 (1 /2) "-7"] \\*

&[1 -1 -1 -1 -1 -1 -1 1 "-9"] \\*

&[1 -1 -1 1 -1 -1 1 1 "-11"] \\*

&[1 2 1 1 1 1 1 2 "-13"] \\*

&[1 1 1 1 1 1 1 1 "-15"]
\tabularnewline \hline 
link6n[1]
&[1 1 1 1 1 1 1 1 "9"] \\*

&[1 1 1 1 -1 1 1 1 "7"] \\*

&[1 1 -1 -1 -1 -1 -1 1 "5"] \\*

&[1 1 1 -1 -2 -1 1 (1 /2) "3"] \\*

&[1 -1 -1 1 1 1 1 1 "1"] \\*

&[1 -1 1 1 1 1 1 1 "-1"]
\tabularnewline \hline 
link7a[1]
&[1 1 1 1 1 1 1 1 1 "10"] \\*

&[1 1 1 1 1 1 1 2 (1 /2) "8"] \\*

&[1 1 1 1 1 -1 4 -1 4 "6"] \\*

&[1 1 1 1 1 -4 -1 1 (1 /4) "4"] \\*

&[1 1 1 1 4 -1 1 -1 4 "2"] \\*

&[-1 -1 -1 -2 1 -1 1 1 (1 /2) "0"] \\*

&[-1 1 -4 -1 1 1 1 1 4 "-2"] \\*

&[-1 2 1 1 1 1 1 1 (1 /2) "-4"] \\*

&[1 1 1 1 1 1 1 1 1 "-6"]
\tabularnewline \hline 
link7a[2]
&[1 1 1 1 1 1 1 1 1 "-2"] \\*

&[1 1 1 1 1 1 1 1 1 "-4"] \\*

&[1 1 1 1 1 -1 4 1 4 "-6"] \\*

&[1 1 1 1 -1 4 -1 1 (1 /4) "-8"] \\*

&[1 -1 1 1 -2 -1 -1 -1 2 "-10"] \\*

&[-1 -1 1 4 -1 -1 1 1 (1 /4) "-12"] \\*

&[-1 -1 2 -1 1 1 1 1 2 "-14"] \\*

&[1 -2 1 1 1 1 1 1 (1 /2) "-16"] \\*

&[1 1 1 1 1 1 1 1 1 "-18"]
\tabularnewline \hline 
link7a[3]
&[1 1 1 1 1 1 1 1 1 "14"] \\*

&[1 1 1 1 1 1 1 2 2 "12"] \\*

&[1 1 1 1 1 1 2 -1 (1 /2) "10"] \\*

&[1 1 1 1 -1 -4 1 -1 4 "8"] \\*

&[1 1 1 -1 -2 1 1 1 (1 /2) "6"] \\*

&[1 1 1 2 -1 -1 1 -1 2 "4"] \\*

&[1 -1 -1 1 -1 -1 1 1 1 "2"] \\*

&[1 -2 1 1 1 1 1 1 2 "0"] \\*

&[1 1 1 1 1 1 1 1 1 "-2"]
\tabularnewline \hline 
link7a[4]
&[1 1 1 1 1 1 1 1 1 "12"] \\*

&[1 1 1 1 1 1 -1 -2 2 "10"] \\*

&[1 1 1 -1 -1 1 -2 -1 (1 /2) "8"] \\*

&[1 -1 1 -1 1 2 1 -1 2 "6"] \\*

&[-1 -1 1 1 4 -1 1 1 (1 /4) "4"] \\*

&[1 1 -1 -2 1 1 1 1 2 "2"] \\*

&[-1 -1 -1 -1 1 1 1 1 1 "0"] \\*

&[-1 -2 1 1 1 1 1 1 2 "-2"] \\*

&[1 1 1 1 1 1 1 1 1 "-4"]
\tabularnewline \hline 
link7a[5]
&[1 1 1 1 1 1 1 1 1 "4"] \\*

&[1 1 1 1 1 1 1 2 (1 /2) "2"] \\*

&[1 1 1 1 1 1 -2 -1 2 "0"] \\*

&[1 1 1 -1 -1 -2 1 1 (1 /2) "-2"] \\*

&[1 1 1 1 -4 -1 1 -1 4 "-4"] \\*

&[1 1 1 -2 1 1 1 1 (1 /2) "-6"] \\*

&[1 1 2 1 1 1 1 1 2 "-8"] \\*

&[-1 -2 1 1 1 1 1 1 (1 /2) "-10"] \\*

&[1 1 1 1 1 1 1 1 1 "-12"]
\tabularnewline \hline 
link7a[6]
&[1 1 1 1 1 1 1 1 1 "14"] \\*

&[1 1 1 1 1 1 -1 2 2 "12"] \\*

&[1 1 1 1 1 1 2 -1 (1 /2) "10"] \\*

&[1 1 1 -1 1 -2 -1 1 2 "8"] \\*

&[1 1 1 -1 2 1 -1 1 (1 /2) "6"] \\*

&[1 1 1 2 1 -1 1 -1 2 "4"] \\*

&[1 -1 1 -1 -1 -1 1 1 1 "2"] \\*

&[1 2 1 -1 1 1 1 1 2 "0"] \\*

&[1 1 1 1 1 1 1 1 1 "-2"]
\tabularnewline \hline 
link7a[7]
&[1 1 1 1 1 1 1 1 1 "7"] \\*

&[1 1 1 1 1 1 -1 2 2 "5"] \\*

&[1 1 1 1 1 -1 4 -1 (1 /4) "3"] \\*

&[1 1 1 1 -1 2 1 1 2 "1"] \\*

&[1 1 -1 1 1 1 1 -1 1 "-1"] \\*

&[1 -1 -1 -8 -1 -1 1 1 8 "-3"] \\*

&[-1 -1 -2 1 1 1 1 1 (1 /2) "-5"] \\*

&[-1 -1 1 1 1 1 1 1 1 "-7"] \\*

&[1 1 1 1 1 1 1 1 1 "-9"]
\tabularnewline \hline 
link7n[1]
&[1 1 1 1 1 1 1 1 1 "-4"] \\*

&[1 1 1 1 -1 1 -1 -1 1 "-6"] \\*

&[1 1 -1 1 1 1 -1 1 1 "-8"] \\*

&[1 -1 1 -1 1 2 1 1 (1 /2) "-10"] \\*

&[1 1 1 1 -1 -1 1 1 1 "-12"] \\*

&[-1 -1 -1 2 -1 1 1 1 (1 /2) "-14"] \\*

&[1 -1 -1 1 1 1 1 1 1 "-16"]
\tabularnewline \hline 
link7n[2]
&[1 1 1 1 1 -1 1 1 1 "0"] \\*

&[1 1 -1 1 -1 1 -1 1 1 "-2"] \\*

&[1 -1 -1 -1 -2 1 -1 1 2 "-4"] \\*

&[-1 1 1 -2 -1 1 1 1 (1 /2) "-6"] \\*

&[-1 -1 -1 -1 1 1 1 1 1 "-8"] \\*

&[-1 -2 -1 1 1 1 1 1 (1 /2) "-10"] \\*

&[1 1 1 1 1 1 1 1 1 "-12"]
\tabularnewline
\hline
\end{longtable}

\subsection{Remarks from the computations}

Consider the trefoil knot (knot $3_{1}$ in the Rolfsen table \cite{Rolfsen2003}).
From the Knot Atlas \cite{Bar-Natan2005} we obtain the following
information about the Khovanov homology groups for this knot, when
homology is computed over $\mathbb{Z}$:

\begin{longtable}{|l|c|}
\hline
\multicolumn{2}{|c|}{{\tablename} \thetable{}: Homology groups for the trefoil}
\addcontentsline{lot}{table}{\thetable{} \quad Homology groups for the trefoil}\\
\hline
\hline
Subcomplex&
Homology\tabularnewline
\hline
\endfirsthead
\hline

\multicolumn{2}{|c|}{{\tablename} \thetable{} -- continued}\\
\hline
\hline
\endhead
 
-1   &  $\mathbb{Z}$ \\
-3   &  $\mathbb{Z}$ \\
-5   &  $\mathbb{Z}$ \\
-7   &  $\mathbb{Z}_{2}$ \\
-9   &  $\mathbb{Z}$ \\
\hline
\end{longtable}

The subcomplex of degree $-7$, has homology $\mathbb{Z}_{2}$, a
torsion group and it contains no free components. In fact over $\mathbb{Q}$,
this subcomplex is acyclic but its Reidemeister torsion (see Section
\ref{sec:R-torsion-knots}) is given by:

\begin{verbatim}
[1 1 1 1 1 "-1"] 
[1 1 1 1 1 "-3"] 
[1 -1 1 1 1 "-5"] 
[1 -2 1 1 (1 /2) "-7"] 
[-1 1 1 1 1 "-9"]
\end{verbatim} 

Note, that the acyclic subcomplex of degree $-7$, recovers the algebraic
torsion exhibited when the homology is calculated over $\mathbb{Z}$.
In fact, the same holds true for all the knots and links in Sections
\ref{sec:R-torsion-knots} and \ref{sec:R-torsion-links}, due to
the choices made for the homology groups.

\end{document}